\documentclass[12pt,a4paper]{amsart}

\usepackage[utf8]{inputenc}
\usepackage[T1]{fontenc}

\usepackage{amscd}
\usepackage{amsfonts}
\usepackage{amsmath}
\usepackage{amssymb}
\usepackage{amsthm}
\usepackage{appendix}
\usepackage{array}
\usepackage{color}
\usepackage{dashrule}
\usepackage{eurosym}
\usepackage{fullpage}
\usepackage{graphicx}
\usepackage{hyperref}

\newcommand{\Q}{\mathbb{Q}}

\newcommand{\Z}{\mathbb{Z}}

\newcommand{\ord}{\operatorname{ord}}
\newcommand{\rank}{\operatorname{rank}}
\newcommand{\sign}{\operatorname{sign}}

\theoremstyle{definition}

\theoremstyle{remark}

\theoremstyle{plain}
\newtheorem{lemma}{Lemma}[section]
\newtheorem{theorem}[lemma]{Theorem}
\newtheorem{corollary}[lemma]{Corollary}

\begin{document}

\author{Tapani Matala-aho}
\author{Louna Sepp\"al\"a}
\address{Matematiikka, PL 8000, 90014 Oulun yliopisto, Finland}
\email{tapani.matala-aho@oulu.fi}
\email{louna.seppala@student.oulu.fi}
\title{Hermite-Thue equation: Pad\'e approximations and Siegel's lemma}
\date{\today}
\thanks{The published version of this article may be found at \url{https://doi.org/10.1016/j.jnt.2018.03.014}.}

\begin{abstract}
Pad\'e approximations and Siegel's lemma are widely used tools in Diophantine approximation theory.
This work has evolved from the attempts to improve Baker-type linear independence measures, either by using the Bombieri-Vaaler version of Siegel's lemma to sharpen the estimates of Pad\'e-type approximations, or by finding completely explicit expressions for the yet unknown 'twin type' Hermite-Pad\'e approximations. 
The appropriate homogeneous matrix equation representing both methods has an $M \times (L+1)$ coefficient matrix, where $M \le L$.
The homogeneous solution vectors of this matrix equation give candidates for the Pad\'e polynomials.    
Due to the Bombieri-Vaaler version of Siegel's lemma, the upper bound of the minimal non-zero solution of the matrix equation can be improved by finding the gcd of all the $M \times M$ minors of the coefficient matrix.
In this paper we consider the exponential function and prove that there indeed exists a big common factor of the $M \times M$ minors, giving a possibility to apply the Bombieri-Vaaler version of Siegel's lemma. 
Further, in the case $M=L$, the existence of this common factor is a step towards understanding the nature of the 'twin type' Hermite-Pad\'e approximations to the exponential function.
\end{abstract}

\maketitle

\section{Introduction}

\subsection{Hermite-Pad\'e approximations}

In 1873 Hermite \cite{Hermite} proved the transcendence of $e$, the base of the natural logarithm. For the proof, Hermite introduced
rational function approximations to the exponential function in the following sense:
Let $l_0,l_1, \ldots,  l_m$ be positive integers and let $\alpha_1, \ldots, \alpha_m$ be distinct complex numbers. 
Denote $\overline{\alpha} = (\alpha_1, \ldots, \alpha_m)^T$, $\overline{l} = (l_0, l_1, \ldots,  l_m)^T$, 
$L_0:=l_0+l_1+\ldots+l_m$, and $L:=l_1+\ldots+l_m$. 
Then there exist polynomials $A_{\overline{l},j}(t, \overline{\alpha}) \in \Q[t, \overline{\alpha}]$ and remainders $R_{\overline{l},j}(t,\overline\alpha)$ such that
\begin{equation}\label{Hermite}
A_{\overline{l}, 0}(t,\overline\alpha) e^{\alpha_j t}  - A_{\overline{l}, j}(t,\overline\alpha) = R_{\overline{l},j}(t,\overline\alpha),
\quad j=1,\ldots,m,
\end{equation}
where
\begin{equation*}
\begin{cases}
\deg_t A_{\overline{l},j}(t,\overline\alpha) \le L_0-l_j,\quad j=0,1,\ldots,m;\\
L_0+1\le \underset{t=0}{\ord} \, R_{\overline{l},j}(t,\overline\alpha) < \infty,\quad j=1,\ldots,m.
\end{cases}
\end{equation*}
Moreover, the coefficients of the polynomial $A_{\overline{l},0}(t,\overline\alpha)$ have explicit expressions in terms of 
the numbers $l_0,l_1, \ldots,  l_m$ and $\alpha_1, \ldots, \alpha_m$ (see \eqref{Anollapolynomi}).
Later these approximations were called \emph{Hermite-Pad\'e approximations}
or \emph{simultaneous Pad\'e approximations of the second type} or \emph{type II Pad\'e approximations}, briefly.

The question of finding explicit Pad\'e approximations to a given set of functions\\ $\{F_1(t), \ldots, F_m(t)\}$ is called a Pad\'e problem. 
Baker and Graves-Morris \cite {BakerGraves-Morris} cover the general setting; 
the problem of simultaneous Pad\'e approximations is stated in Chapter 8 of their book.
In case $l_0=l_1=\ldots=l_m=:l$ for $l\in \mathbb{Z}_{\ge 1}$, we use the term \emph{diagonal (Pad\'e approximations)}.
Diagonal Hermite-Pad\'e approximations of the generalised hypergeometric series are quite well established;
see de Bruin \cite{BRUIN}, Huttner \cite{Huttner}, Matala-aho \cite{MAT2011}, Nesterenko \cite{Nesterenko}, and Niki\v{s}in \cite{Nikisin} for more details.
In this work we shall not pursue further this general Pad\'e problem but focus on the Pad\'e approximations to the exponential function instead.

\subsection{The twin problem}\label{sec:twin}

The problem of finding explicit type II Hermite-Pad\'e approximations to the exponential series in the case where the degrees of 
the polynomials are free parameters and the orders of the remainders are equal was, as mentioned, resolved already by Hermite.
But the twin problem (as we call it) of finding explicit type II Hermite-Pad\'e approximations in the case where the degrees of the polynomials are 
the same but the orders of the remainders are free parameters is still open.

Let now $l_1, \ldots,  l_m$ be positive integers and let $\alpha_1, \ldots, \alpha_m$ be distinct variables. 
Denote $\overline{\alpha} = (\alpha_1, \ldots, \alpha_m)^T$, $\overline{l} = (l_1, \ldots,  l_m)^T$, and $L:=l_1+\ldots+l_m$. 
Then we may state the twin problem as follows:
Find an explicit polynomial $B_{\overline{l},0} (t,\overline\alpha)$, polynomials
$B_{\overline{l},j}(t, \overline{\alpha})$ and remainders $S_{\overline{l},j}(t,\overline\alpha)$, $j=1,\ldots,m$, satisfying the equations
\begin{equation}\label{TwinHermite}
B_{\overline{l},0} (t,\overline\alpha)  e^{\alpha_j t} - B_{\overline{l},j} (t,\overline\alpha)  =
S_{\overline{l},j} (t,\overline\alpha) , \quad j=1, \ldots, m,
\end{equation}
with
\begin{equation*}
\begin{cases}
\deg_t B_{\overline{l},j} (t,\overline\alpha)  \le L, &j=0, \ldots, m; \\
L+l_j +1 \le \underset{t=0}{\ord} \, S_{\overline{l},j}(t,\overline\alpha) < \infty, &j=1, \ldots, m.
\end{cases}
\end{equation*}
In the diagonal case, the twin approximations \eqref{TwinHermite} and the classical Hermite-Pad\'e approximations \eqref{Hermite} coincide. The fact that Pad\'e approximations to a given function are unique up to a non-zero constant is expressed in the homogeneous vector specified by the coefficients of the denominator polynomial $A_{\overline{l},0}(t)$ or $B_{\overline{l},0}(t)$.

Motivation for finding the explicit twin approximations \eqref{TwinHermite} comes from their possible applicability
to arithmetical questions. The known Hermite-Pad\'e approximations \eqref{Hermite} are well suited e.g. for proving 
sharp transcendence measures for rational powers of $e$ (see \cite{EMS2016}).
On the other hand, the following type Pad\'e  approximations
\begin{equation}\label{DualHermite}
B_{\overline{\nu},0} (t,\overline\alpha) e^{\alpha_j t} - B_{\overline{\nu},j} (t,\overline\alpha) = 
S_{\overline{\nu},j} (t,\overline\alpha), \quad j=1, \ldots, m,
\end{equation}
where $\overline{\nu} = ( \nu_1, \ldots, \nu_m)^T\in\mathbb{Z}_{\ge 1}^m$, $\nu_1\le l_1,\ldots,\nu_m\le l_m$,
$\nu_1 + \ldots + \nu_m =: M \le L$ and
\begin{equation*}
\begin{cases}
\deg_t B_{\overline{\nu},j} (t,\overline\alpha) \le L, &j=0, \ldots, m; \\
L+\nu_j +1 \le \underset{t=0}{\ord} \, S_{\overline{\nu},j} (t,\overline\alpha) < \infty, &j=1, \ldots, m
\end{cases}
\end{equation*}
(of which \eqref{TwinHermite} is a special case) have been successfully used in proving a Baker-type transcendence measure for $e$ (see \cite{AKT}). In Baker-type bounds the dependence on the individual heights of the coefficients of the linear form is visible in the bound. This provides an additional challenge compared to settling for the maximum height only.

Due to the lack of explicit twin approximations \eqref{TwinHermite}, 
the works considering Baker-type lower bounds for linear forms have relied on Siegel's lemma; see Baker \cite{BAKER1965},
Mahler \cite{MAHLER1975}, and, for a generalised transcendence measure of $e$, Ernvall-Hyt\"onen et al. \cite{AKT}.
In these Baker-type lower bounds the error terms are weaker than the 
corresponding ones in those transcendence measures that depend on the maximum height only
(see Hata \cite{HATA1995} and Ernvall-Hyt\"onen et al. \cite{EMS2016}).

For example, in \cite{AKT} the authors present an explicit Baker-type lower bound
$$
\left| \beta_0 + \beta_1 e + \beta_2 e^2 + \ldots + \beta_m e^m \right| > \frac{1}{h^{1+ \epsilon (h)}}
$$
valid for all
$$
\beta = (\beta_0, \ldots, \beta_m)^T \in \mathbb{Z}^{m+1}, \quad h_i = \max \{ 1, | \beta_i | \}, \quad h=h_1 \cdots h_m,
$$
with a lower bound for $\log h$ depending on $m$ and an error term
\begin{equation}\label{errorterm}
\epsilon (h) = \frac{(4+7m) \sqrt{\log (m+1)}}{\sqrt{\log \log h}}.
\end{equation}
On the other hand, in \cite{EMS2016} the authors prove a transcendence measure for $e$ based on explicit Pad\'e approximations and there the corresponding error term
$$
\frac{c m^2 \log m}{\log \log H}, \quad c=c(m) \le 1, \quad H = \max_{1 \le i \le m} \{ 1, | \beta_i | \},
$$
is smaller than $m \epsilon (H^m)$ (in the above works $\log H$ is very large compared to $m$).
Thus, it is reasonable to suppose that explicit twin approximations \eqref{TwinHermite} could perhaps imply a sharper Baker-type transcendence measure for $e$.

\subsection{Siegel's lemma}

If we write
\begin{equation}\label{Bnu0}
B_{\overline{\nu},0}(t,\overline\alpha) = \sum_{h=0}^L c_h \frac{L!}{h!} t^h,
\end{equation}
then \eqref{DualHermite} yields a group of $M$ equations with integer coefficients in the $L+1$ unknowns $c_h$.
The solutions to groups of equations with less equations than unknowns can be estimated by using the Thue-Siegel lemma.
In 1909 Thue \cite{Thue1909} improved the Liouville bound for algebraic numbers. 
For that purpose he needed to find a small non-zero integer solution $(x_1, \ldots ,x_N)$ to a system of $M$ equations 
in integer coefficients with $N$ unknowns, $M < N$.
An essential feature is that the small solution is bounded with a non-trivial upper bound depending on the coefficients. 
Thue's idea was present already in the 1908 paper \cite{Thue1908A}.

In his celebrated paper \cite{Siegel} Siegel formalised Thue's idea which is known since then as Thue-Siegel's lemma or Siegel's lemma.
Below we present a practical variant of Siegel's lemma from Mahler \cite{MAHLER1975}. Throughout this paper $\mathcal{M}_{k \times l} (R)$ denotes the set of $k \times l$ matrices with coefficients in $R$.
 
\begin{lemma}\label{Thue-Siegel's}\index{Siegel's lemma} 
Let $\mathcal{V} = (v_{mn}) \in \mathcal{M}_{M \times N} (\Z)$, and assume that
\begin{equation}\label{rivisumma}
\|\underline{v}_m\|_{1}:=\sum_{n=1}^{N}|v_{mn}|\in \mathbb{Z}_{\ge 1},\quad m=1,\ldots,M,
\end{equation}
where $\underline{v}_m$ denotes the $m$th row of the matrix $\mathcal{V}$.
Suppose that $M<N$; then the equation 
\begin{equation*}\label{SIEGELGROUP}
\mathcal{V} \overline{x} = \overline{0}
\end{equation*}
has a non-zero integer solution $\overline{x}=(x_1,\ldots,x_N)^T \in \mathbb{Z}^{N}\setminus\{\overline{0}\}$ with
\begin{equation*}\label{SIEGELESTIMATE}
1\le \left\| \overline{x} \right\|_\infty:= \max_{1\le n\le N} |x_n|\le 
\left\lfloor (\|\underline{v}_1\|_{1}\cdots \|\underline{v}_M\|_{1})^{\frac{1}{N-M}}\right\rfloor.
\end{equation*}
\end{lemma}

It had been noticed that explicit Pad\'e approximations were not always enough for Diophantine purposes.
Therefore, starting from Siegel's work \cite{Siegel}, Siegel's lemma was greatly appreciated
because of its flexibility and power in Diophantine approximation and transcendence proofs.

The use of Siegel's lemma is the reason for introducing the above mentioned parameters $\nu_1,\ldots,\nu_m$ 
into the approximation problem \eqref{DualHermite}. 
Since we don't know the explicit solution to \eqref{TwinHermite}, the parameters $\nu_1,\ldots,\nu_m$  give more freedom to optimise the estimate. 
In case \eqref{Hermite} the explicit solution is known and therefore the problem is presented only with the parameters 
$l_0,l_1, \ldots,  l_m$.

As explained in Section \ref{sec:twin}, the lower bounds for linear forms 
coming from Siegel's lemma are not always as sharp as those coming from Pad\'e approximations.
Therefore it is plausible that the already existing results might be improved by
some refined version of Siegel's lemma. Such an improvement indeed exists and is given in Bombieri and Vaaler \cite{BOMVAA};
for a shorter proof in the integer case, see \cite[Theorem 14.3]{MAT2016}. 
Below we present the Bombieri-Vaaler version in the integer case.

\begin{lemma}\label{Bombieri-Vaaler} \cite[Theorem 2]{BOMVAA}
Let $\mathcal{V} \in \mathcal{M}_{M \times N} (\Z)$, $M < N$, and $\rank \mathcal{V} = M$. Then the equation
\begin{equation*}\label{BomVaamatrix}
\mathcal{V} \overline{x} = \overline{0}
\end{equation*}
has $N-M$ linearly independent integer solutions $\overline{x}_1, \ldots, \overline{x}_{N-M} \in \Z^{N} \setminus \{ \overline{0} \}$ such that
$$
\left\| \overline{x}_1 \right\|_\infty \cdots \left\| \overline{x}_{N-M} \right\|_\infty \le \frac{\sqrt{\det (\mathcal{V} \mathcal{V}^T )}}{D},
$$
where $D$ is the greatest common divisor of all the $M \times M$ minors of $\mathcal{V}$.
\end{lemma}
In fact  Bombieri and Vaaler \cite{BOMVAA} proved a more general result over the algebraic numbers by using geometry of 
numbers over the ad\`eles; see also Fukshansky \cite{Fukshansky2006}.

An important feature of Siegel's lemma is that it only implies the existence of a non-trivial integer solution
with an appropriate upper bound. 
A priori, no information about the explicit expressions of the solutions to, say, equations \eqref{Hermite} or \eqref{TwinHermite} is given.
However, there is a deep connection between Siegel's lemma and the solution, as revealed in Bombieri and Vaaler's version of 
Siegel's lemma, Lemma \ref{Bombieri-Vaaler}.
Namely, the Grassmann coordinates of the exterior product of the row vectors of the matrix $\mathcal{V}$ are precisely the 
$M\times M$ minors of $\mathcal{V}$; for details, see \cite[Section 14.2]{MAT2016}. 
On the other hand, the solution to the Pad\'e problem \eqref{TwinHermite} is a homogeneous vector of the 
$L\times L$ minors of the coefficient matrix of the group of equations formed from \eqref{DualHermite} and \eqref{Bnu0} in the case $\nu_j=l_j$, $M=L$.
So the same minors that one needs to study to be able to apply Lemma \ref{Bombieri-Vaaler} actually form the solution vector to the 
Pad\'e problem.
Hence the title \emph{Hermite-Thue equation}.

Since Hermite's, Thue's and Siegel's works, Pad\'e approximations and Siegel's lemma have been cornerstones in transcendence theory and in the theory of Diophantine approximations.
The book of Fel'dman and Nesterenko \cite{Feldman} is recommended reading for those interested in a more detailed overview of transcendental number theory.
See also \cite{BAKER1975}, \cite{SHIDLOVSKII}, \cite{SIEGEL}.

\section{Results}

As a warm-up we resolve the type II  Pad\'e approximation problem \eqref{Hermite}, the tame case, in Section \ref{sec:tame} by computing the homogeneous vector of the $L\times L$ minors of the coefficient matrix (Cramer's rule).
 
In Section \ref{sec:wild} we examine our case \eqref{DualHermite}, the wild case, which turns out to be much tougher. 
Choose $\overline{\alpha} = \overline{a} := (a_1, \ldots, a_m)^T$, where $a_1, \ldots, a_m$ are pairwise different, non-zero integers. Write
\begin{equation*}
B_{\overline{\nu},0}(t,\overline{a}) = \sum_{h=0}^L c_h \frac{L!}{h!} t^h,
\end{equation*}
where $\overline{\nu} := (\nu_1, \ldots, \nu_m )^T$, and the numbers
$\nu_1, \ldots, \nu_m, l_1, \ldots, l_m \in \Z_{\ge 1}$ satisfy
$$
1 \le \nu_j \le l_j, \quad M := \nu_1 + \ldots + \nu_m \le L := l_1 + \ldots + l_m.
$$
Then \eqref{DualHermite} yields the matrix equation
\begin{equation*}\label{unknowneq1}
\mathcal{V} \overline{c} = \overline{0}, \quad \overline{c} := (c_0, c_1, \ldots, c_L)^T,
\end{equation*}
with
\begin{equation*}
\mathcal{V} = \mathcal{V}(\overline{a}) :=
\begin{pmatrix}
\binom{L + 1}{0} a_1^L & \binom{L + 1}{1} a_1^{L-1} & \cdots & \binom{L + 1}{L - 1} a_1 & \binom{L +1}{L} \\
\binom{L + 2}{0} a_1^L & \binom{L + 2}{1} a_1^{L-1} & \cdots & \binom{L + 2}{L - 1} a_1 & \binom{L +2}{L} \\
\vdots & \vdots & \ddots & \vdots & \vdots \\
\binom{L + \nu_1}{0} a_1^L & \binom{L + \nu_1}{1} a_1^{L-1} & \cdots & \binom{L + \nu_1}{L - 1} a_1 & \binom{L +\nu_1}{L} \\

\vdots & \vdots & \ddots & \vdots & \vdots \\

\vdots & \vdots & \ddots & \vdots & \vdots \\

\binom{L + 1}{0} a_m^L & \binom{L + 1}{1} a_m^{L-1} & \cdots & \binom{L + 1}{L - 1} a_m & \binom{L +1}{L} \\
\binom{L + 2}{0} a_m^L & \binom{L + 2}{1} a_m^{L-1} & \cdots & \binom{L + 2}{L - 1} a_m & \binom{L +2}{L} \\
\vdots & \vdots & \ddots & \vdots & \vdots \\
\binom{L + \nu_m}{0} a_m^L & \binom{L + \nu_m}{1} a_m^{L-1} & \cdots & \binom{L + \nu_m}{L - 1} a_m & \binom{L +\nu_m}{L} \\
\end{pmatrix}_{\!M \times (L+1)},
\end{equation*}
for which Siegel's lemma produces a non-zero integer solution with a non-trivial upper bound.

In order to be able to study the common factors of the minors of the integer matrix $\mathcal{V}(\overline{a})$, we switch to the corresponding polynomial matrix $\mathcal{V}(\overline{\alpha}) \in \mathcal{M}_{M \times (L+1)} (\mathbb{Z}[\alpha_1, \ldots, \alpha_m])$.
We are able find two different high order polynomial factors from the $M \times M$ minors of $\mathcal{V}(\overline{\alpha})$ (see Theorem \ref{polynomialfactortheorem}). Choosing $\overline{\alpha} = \overline{a} = (a_1, \ldots, a_m)^T$ then leads us to
\begin{theorem}\label{IntegerFactortheorem}
Let $a_1,\ldots,a_m \in \Z$. Then
\begin{equation}\label{integerfactor}
\left( \prod_{1 \le j \le m} a_j^{\binom{\nu_j}{2}} \right)
\prod_{1 \le i < j \le m} (a_i - a_j)^{\min \left\{ \nu_i^2,\, \nu_j^2 \right\}} 
\underset{\mathbb{Z}}{\bigg|} D(\overline{a}),
\end{equation}
where $D(\overline{a})$ is the greatest common divisor of all the $M \times M$ minors of the matrix $\mathcal{V}(\overline{a}) \in \mathcal{M}_{M \times (L+1)} (\mathbb{Z})$.
\end{theorem}

In addition, we can prove that the rank of the polynomial matrix $\mathcal{V}(\overline{\alpha})$ over $\Z[\alpha_1, \ldots, \alpha_m]$ is $M$ (see Lemma \ref{ranklemma2}), 
but the rank of the coefficient matrix $\mathcal{V}(\overline{a})$ over the integers remains an open question. This is a problem that needs answering before the Bombieri-Vaaler version of Siegel's lemma (Lemma \ref{Bombieri-Vaaler}) can be applied. At the moment it seems to be hard and requires some further investigation. Assuming that the rank condition could be fulfilled, the common factor of Theorem \ref{IntegerFactortheorem} gives us the possibility to improve the error term \eqref{errorterm}.

In the particular case $\nu_j = l_j$, $M=L$, the wild case \eqref{DualHermite} reduces to the twin Pad\'e problem \eqref{TwinHermite}, for which we give a partial answer in the following theorem:

\begin{theorem}\label{UnKnown}
Let $l_1, \ldots,  l_m$ be positive integers and let $\alpha_1, \ldots, \alpha_m$ be distinct variables. 
Denote $\overline{\alpha} = (\alpha_1, \ldots, \alpha_m)^T$, $\overline{l} = (l_1, \ldots,  l_m)^T$, $L:=l_1+\ldots+l_m$. 
Then there exist non-zero polynomials $B_{\overline{l},j}(t, \overline{\alpha}) \in \Q[t, \overline{\alpha}]$ and remainders 
$S_{\overline{l},j}(t,\overline\alpha)$ such that
\begin{equation*}\label{TWIN2}
B_{\overline{l},0}(t, \overline{\alpha}) e^{\alpha_j t} - B_{\overline{l},j}(t, \overline{\alpha})  = 
S_{\overline{l},j}(t, \overline{\alpha}) , \quad j=1, \ldots, m,
\end{equation*}
where 
\begin{equation*}
\begin{cases}
\deg_t B_{\overline{l},j} (t, \overline{\alpha})  \le L, &j=0, \ldots, m; \\
L+l_j +1 \le \underset{t=0}{\ord} \, S_{\overline{l},j}(t, \overline{\alpha})  < \infty, &j=1, \ldots, m.
\end{cases}
\end{equation*}
Moreover, we have 
\begin{equation*}\label{}
B_{\overline{l},0}(t,\overline{\alpha}) = \sum_{i=0}^L \frac{L!}{i!} \tau_{i}\!\left(\overline{l},\overline{\alpha}\right) t^i,
\quad \tau_{i}\!\left(\overline{l},\overline{\alpha}\right) = \frac{(-1)^i\mathcal{V}[i]}{T\!\left(\overline{l},\overline{\alpha}\right)} \in \Z[\overline{\alpha}],                  
\end{equation*}
where
\begin{equation}\label{Wildfactor}
T\!\left(\overline{l},\overline{\alpha}\right) := \alpha_1^{\binom{l_1}{2}} \cdots \alpha_m^{\binom{l_m}{2}}
\prod_{1 \le i < j \le m} (\alpha_i - \alpha_j)^{\min \left\{ l_i^2,\, l_j^2 \right\}}
\end{equation}
is a common factor of the $L \times L$ minors $\mathcal{V}[i]$, $i=0,\ldots,L$, of the matrix $\mathcal{V}(\overline{\alpha})$.
\end{theorem}

The computational evidence in
%Appendix
\ref{minorappendix} indicates that the common factor \eqref{Wildfactor} of Theorem \ref{UnKnown} may be the best achievable 
primitive polynomial factor.
As opposite to the tame case \eqref{Hermite} however, after dividing by the common factor, we are now left with the non-explicit polynomials $\frac{L!}{i!} \tau_{i}\!\left(\overline{l},\overline{\alpha}\right)$, $i=0,1,\ldots,L$, representing the coordinates of the homogeneous solution vector. 

There is plenty of literature written about determinants, but very few sources mention the kinds block matrices from which our minors originate.
The works \cite{FloweHarris}, \cite{Krattenthaler} and \cite{ALF1976}, however, gave the basic inspiration to our considerations
in the tame case. In the wild case \eqref{DualHermite} the determinants were painful to open, and despite of finding 
the factor \eqref{Wildfactor}, still remain unfinished. 
We have not found any reference elsewhere addressing this problem. 

We stress that giving rigorous proofs of products of several high order factors means that we need to work in the polynomial ring 
$\Z[\alpha_1, \ldots, \alpha_m]$ which is a UFD and where e.g. $\alpha_i$ and $\alpha_i-\alpha_j$, $i\ne j$, are prime elements. 
Later on, when we specialise the variables $\alpha_1, \ldots, \alpha_m$ to be integers, we get the corresponding factors in $\Z$.

Finally, we wish to point out that Section \ref{sec:vandermonde} concerning factors of Vandermonde-type block determinants---although regarded mainly as our tools in this paper---could be of interest on its own to readers from various fields.

\section{Preliminaries and tools}

\subsection{Exterior algebras}

The properties of exterior algebras provide us useful tools for studying determinants. Some essentials are covered here---for a deeper insight we recommend Rotman \cite{Rotman}.

Let $\mathbf{R}$ be a commutative ring and $\mathbf{M}$ a free $\mathbf{R}$-module with $\rank \mathbf{M} = n$. The exterior algebra of $\mathbf{M}$ is denoted by $\bigwedge (\mathbf{M})$. 
The ring product $\wedge$ is called \emph{wedge product} and it has the property $m \wedge m = 0$ for all $m \in \mathbf{M}$.

\subsubsection{Increasing lists}

Let $n \in \Z_{\ge 1}$ and $p \in \{ 0, 1, \ldots, n \}$. The numbers
$$
i_1, i_2, \ldots, i_p \in \{ 1, \ldots, n \}
$$
form an \emph{increasing $(0 \le p \le n)$-list} $\sigma_p$, if
$$
1 \le i_1 < i_2 < \ldots < i_p \le n.
$$
If $p=0$, then $\sigma_p = \varnothing$. The set of all increasing $(0 \le p \le n)$-lists is denoted by
$$
C(n, p) := \{ \{ i_1, i_2, \ldots, i_p \} \; | \; 1 \le i_1 < i_2 < \ldots < i_p \le n \}.
$$\

Let $H = \{ h_1, \ldots, h_p \}$ and $K = \{ k_1, \ldots, k_q \}$ be increasing $(0 \le p \le n)$- and $(0 \le q \le n)$-lists, respectively. 
When the lists $H$ and $K$ are disjoint, we denote by $\tau_{H, K}$ the permutation in the symmetric group $S_{p+q}$ which arranges the list
$$
h_1, \ldots, h_p, k_1, \ldots, k_q
$$
into an increasing $(0 \le p+q \le n)$-list
$$
H*K = \{ j_1, \ldots, j_{p+q} \}, \quad 1 \le j_1 < \ldots < j_{p+q} \le n.
$$
More generally, let $H_k = \{ h_{k,1}, \ldots, h_{k,n_k} \}$, $k=1, \ldots, m$, be $m$ increasing $(0 \le n_k \le n)$-lists.
When the lists $H_k$ are pairwise disjoint, we denote by $\tau_{H_1,\ldots, H_m}$ the permutation in the symmetric group $S_{n_1+\ldots +n_m}$ which arranges the list
$$
h_{1,1}, \ldots, h_{1,n_1}, h_{2,1}, \ldots, h_{2,n_2}, \ldots, h_{m,1}, \ldots, h_{m,n_m}
$$
into an increasing $(0 \le n_1+\ldots+n_m \le n)$-list
$$
H_1*\ldots *H_m = \{ j_1, \ldots, j_{n_1+\ldots+n_m} \}, \quad 1 \le j_1 < \ldots < j_{n_1+\ldots+n_m} \le n.
$$

\subsubsection{Basis vectors}

When $I = \{ i_1, \ldots, i_p \}$ is an increasing $(0 \le p \le n)$-list, we use the notation
$$
\overline{A}_I := \overline{a}_{i_1} \wedge \ldots \wedge \overline{a}_{i_p},
$$
where $\overline{a}_{i_1}, \ldots, \overline{a}_{i_p} \in \mathbf{M}$. The element $\overline{A}_I$ is called a \emph{$p$-vector}.

Let $\{\overline{e}_1, \ldots, \overline{e}_n\}$ be a basis of $\mathbf{M}$. For $p \in \{ 0, 1, \ldots, n \}$, consider the products
$$
\overline{E}_{\sigma_p} := \overline{e}_{i_1} \wedge \overline{e}_{i_2} \wedge \ldots \wedge \overline{e}_{i_p}, 
\quad \sigma_p =\{ i_1, i_2, \ldots, i_p \} \in C(n, p),
$$
with $\overline{E}_\varnothing = 1$. 
There are $\binom{n}{p}$ of them for each $p \in \{0, 1, \ldots, n\}$, and together they form a basis for $\bigwedge (\mathbf{M})$.

Let $H$ and $K$ be increasing $(0 \le p \le n)$- and $(0 \le q \le n)$-lists, respectively. Then
$$
\overline{E}_H \wedge \overline{E}_K =
\begin{cases}
0, &\text{if}\ H \cap K \neq \varnothing; \\
\sign (\tau_{H,K}) \overline{E}_{H*K}, &\text{if}\ H \cap K = \varnothing.
\end{cases}
$$

\subsubsection{Grassmann coordinates}

When $A$ is an $s \times t$ matrix and $H \in C(s,p)$, $K \in C(t,q)$, we denote by
$$
A_{HK} := (a_{hk}), \quad h \in H, \quad k \in K,
$$
the $p \times q$ submatrix of $A$ with rows and columns chosen according to the lists $H$ and $K$, respectively.

Let $\overline{a}_{i_1}, \ldots, \overline{a}_{i_p} \in \mathbf{M}$, where $\{ i_1, \ldots, i_p \} =: I$ is an increasing $(0 \le p \le n)$-list. 
Then
$$
\overline{a}_{i_1} \wedge \ldots \wedge \overline{a}_{i_p} = \sum_{L \in C(n,p)} \det \left(A_{L, I} \right) \overline{E}_L,
$$
where
$$
\det \left(A_{L, I} \right) = \det(a_{l,i}), \quad l \in L, \quad i \in I,
$$
is a $p \times p$ minor of the $n \times p$ matrix formed by the vectors $\overline{a}_{i_1}, \ldots, \overline{a}_{i_p}$. 
The determinants $\det \left(A_{L,I} \right)$ are the so-called \emph{Grassmann or Pl\"ucker coordinates} of the 
$p$-vector $\overline{a}_{i_1} \wedge \ldots \wedge \overline{a}_{i_p}$.

Let $I = \{ i_1, \ldots, i_p \}$ and $J = \{ j_1, \ldots, j_q \}$ be increasing $(0 \le p \le n)$- and $(0 \le q \le n)$-lists, respectively. 
Then it is immediate that 
$$
\overline{A}_I \wedge \overline{A}_J = 
\sum_{H \in C(n,p), K \in C(n,q)} \sign (\tau_{H, K}) \det \left( A_{H, I} \right) \det \left( A_{K, J} \right) \overline{E}_{H * K}.
$$
This  product of a $p$-vector and a $q$-vector has a direct generalisation given below. 

\begin{lemma}\label{productlemma2}
Let $I_k = \{ i_{k,1}, \ldots, i_{k,n_k} \}$, $k=1, \ldots, m$, be $m$ increasing $(0 \le n_k \le n)$-lists.
Then
\begin{equation*}
\bigwedge_{k=1}^m \overline{A}_{I_k} =
\sum_{H_1 \in C(n,n_1), \ldots, H_m \in C(n,n_m)} \sign (\tau_{H_1, \ldots, H_m}) \det \left( A_{H_1, I_1} \right) 
\cdots \det \left( A_{H_m, I_m} \right) \overline{E}_{H_1 * \ldots * H_m}.
\end{equation*}
\end{lemma}

\subsection{Generalised minor expansion}\label{Section2.2}

Since the Grassmann coordinates are determinants, we may use them to prove determinant expansion formulas.

\begin{lemma}\label{manyblocksexpansion}
Let $A=(a_{ij}) \in \mathcal{M}_{n \times n}(\mathbf{R})$, where $\mathbf{R}$ is a commutative ring. 
Let $I_k = \{ i_{k,1}, \ldots, i_{k,n_k} \}$, $k=1, \ldots, m$, be $m$ increasing $(0 \le n_k \le n)$-lists such that $\sum_{k=1}^m n_k=n$ and $I_j \cap I_k = \varnothing$ when $j \neq k$. Then
\begin{equation*}
\det (A)=\sign (\tau_{I_1, \ldots, I_m}) \sum_{\substack{H_1 \in C(n,n_1), \ldots, H_m \in C(n,n_m) \\ H_i \cap H_j = \varnothing, i \neq j}} \sign(\tau_{H_1, \ldots, H_m}) \det(A_{H_1,I_1}) \cdots \det (A_{H_m,I_m}).
\end{equation*}
\end{lemma}

\begin{proof}
First
\begin{equation*}
\begin{split}
\bigwedge_{k=1}^m \left( \overline{a}_{i_{k,1}} \wedge \ldots \wedge \overline{a}_{i_{k,n_k}} \right) &= \bigwedge_{k=1}^m \left(A \overline{e}_{i_{k,1}} \wedge \ldots \wedge A \overline{e}_{i_{k,n_k}} \right) \\
&= \det (A) \cdot \bigwedge_{k=1}^m \left( \overline{e}_{i_{k,1}} \wedge \ldots \wedge \overline{e}_{i_{k,n_k}} \right) \\
&= \det (A) \sign (\tau_{I_1, \ldots, I_m}) \cdot \overline{e}_1 \wedge \ldots \wedge \overline{e}_n.
\end{split}
\end{equation*}
On the other hand, Lemma \ref{productlemma2} implies
\begin{multline*}
\bigwedge_{k=1}^m \left( \overline{a}_{i_{k,1}} \wedge \ldots \wedge \overline{a}_{i_{k,n_k}} \right) = \\
\sum_{\substack{H_1 \in C(n,n_1), \ldots,\\ H_m \in C(n,n_m)}} \sign (\tau_{H_1, \ldots, H_m}) \det \left( A_{H_1, I_1} \right) \cdots \det \left( A_{H_m, I_m} \right) \overline{E}_{H_1 * \cdots * H_m}.
\end{multline*}
Since $\sum_{k=1}^m n_k=n$, we have $\overline{E}_{H_1 * \cdots * H_m} = \overline{e}_1 \wedge \ldots \wedge \overline{e}_n$ for all 
$H_k \in C(n,n_k)$, $k=1, \ldots, m$, such that $H_i \cap H_j = \varnothing$ when $i \neq j$. Hence
$$
\det (A) \sign (\tau_{I_1, \ldots, I_m}) = \sum_{\substack{H_1 \in C(n,n_1), \ldots, H_m \in C(n,n_m) \\ H_i \cap H_j = \varnothing, i \neq j}} \sign(\tau_{H_1, \ldots, H_m}) \det(A_{H_1,I_1}) \cdots \det (A_{H_m,I_m}),
$$
and multiplication by $\sign (\tau_{I_1, \ldots, I_m})$ proves the claim.
\end{proof}

\subsection{Polynomial rings}

We are about to prove the existence of high order factors in generalised Vandermonde-type polynomial block determinants. 
For the proofs to work, we need to study the determinants inside a polynomial ring. 
If the ring of coefficients is not a field, then the division algorithm is not available, but luckily the following well-known lemma over a field  is valid also over any integral domain of characteristic zero.

\begin{lemma}\label{derivaattalemma}
Let $\mathbf{I}$ be an integral domain of characteristic zero, 
$P(x) \in \mathbf{I}[x]$, $x_0 \in \mathbf{I}$ and $n \in \Z_{\ge 1}$. If
\begin{equation}\label{derivaattaehto2}
P^{(k)} (x_0) = 0 \quad \text{for all} \ k \in \{0, 1, \ldots, n-1\},
\end{equation}
then
$$
(x-x_0)^n \underset{\mathbf{I}[x]}{\bigg|} P(x).
$$
\end{lemma}

\begin{proof}
Using the binomial expansion we may write
$$
P (x) = P (x - x_0+ x_0) = a_0 + a_1 (x - x_0) + \ldots + a_d (x - x_0)^d,
$$
where
$$
a_i \in \mathbf{I}, \quad i = 0, 1, \ldots, d, \quad d := \deg P(x).
$$
Now condition \eqref{derivaattaehto2} implies $a_0 = a_1 = \ldots = a_{n-1} = 0$, so
\begin{equation*}
(x - x_0)^n \underset{\mathbf{I}[x]}{\bigg|} P(x).
\end{equation*}
\end{proof}

\section{Generalised Vandermonde-type polynomial block matrices}\label{sec:vandermonde}

Let $\mathbf{D}$ be an integral domain of characteristic zero. 
In the following we shall work in the polynomial ring denoted by 
$\mathbf{D}[x_1,\ldots,x_m]$ which is a free commutative $\mathbf{D}$-algebra generated by $\{x_1,\ldots,x_m\}$.
(As usual, the elements $x_1,\ldots,x_m$ are called indeterminates or variables.)
Note also that if $\mathbf{D}$ is a UFD, then $\mathbf{D}[x_1,\ldots,x_m]$ is also UFD, where $x_1,\ldots,x_m$ are prime elements.
Further, a polynomial $P \in \mathbf{D}[x_1, \ldots, x_m]$ is called \emph{primitive}, if the $\gcd$ of its coefficients is $1$.
By setting $\deg 0(x):=-\infty$ for the zero-polynomial $0(x)\in\mathbf{D}[x]$,
the degree formula $\deg (a(x)b(x))= \deg a(x) + \deg b(x)$ holds for all single variable polynomials $a(x), b(x) \in \mathbf{D}[x]$.

We shall use the notation $(x_1, \ldots, \hat{x}_i, \ldots, x_m)$ for the $(m-1)$-tuple where the component $x_i$ has been left out.
Before going further, let us recall the definition of multinomials:
$$
\binom{k}{k_1, \ldots, k_m}:=\frac{k!}{k_1!\cdots k_m!},\quad k_1+\ldots +k_m=k,\quad k_1,\ldots, k_m,k\in \mathbb{Z}_{\ge 0}.
$$

Let $m\in \Z_{\ge 1}$. Pick then $m$ positive integers $n_1, n_2, \ldots, n_m \in \Z_{\ge 1}$ and set $n:=n_1 + n_2 + \ldots + n_m$.
Let $p_0(x), p_1(x), \ldots, p_{n-1}(x) \in \mathbf{D}[x]$ be $n$ single variable polynomials.

\subsection{Case A}

Our method is based on Lemma \ref{derivaattalemma} along with the idea already used e.g. by Flowe and Harris in \cite{FloweHarris}.
However, no combinatorial argument will be needed in our proof, and there are no restrictions on the polynomials.

Denote
$$
A_j :=
\begin{pmatrix}
p_0 (x_j) & p_1 (x_j) & \cdots & p_{n-1} (x_j) \\
p_0' (x_j) & p_1' (x_j) & \cdots & p_{n-1}' (x_j) \\
\vdots & \vdots & \ddots & \vdots \\
p_0^{(n_j -1)} (x_j) & p_1^{(n_j -1)} (x_j) & \cdots & p_{n-1}^{(n_j -1)} (x_j)
\end{pmatrix}_{\!n_j \times n},
\quad j=1, \ldots, m,
$$
and let
\begin{equation*}\label{polymatrix}
\mathcal{A} :=
\begin{pmatrix}
A_1 \\
A_2 \\
\vdots \\
A_m
\end{pmatrix}_{\!n \times n}.
\end{equation*}

\begin{lemma}\label{blockpolyvande}
With the above notation, we have
$$
\prod_{1 \le i < j \le m} (x_i - x_j)^{n_i n_j} \underset{\mathbf{D}[x_1, \ldots, x_m]}{\bigg|} \det \mathcal{A}.
$$
\end{lemma}

\begin{proof}
Let us write
\begin{equation}\label{rivi}
\overline{r}^{(k)} (x) :=
\begin{pmatrix}
\left( \frac{\mathrm{d}}{\mathrm{d}x} \right)^k p_0 (x) & \left( \frac{\mathrm{d}}{\mathrm{d}x} \right)^k p_1 (x) 
& \cdots & \left( \frac{\mathrm{d}}{\mathrm{d}x} \right)^k p_{n-1} (x)
\end{pmatrix}^T,
\quad \overline{r}_j^{(k)} := \overline{r}^{(k)} (x_j),
\end{equation}
for $k \in \Z_{\ge 0}$, so that
$$
\det \mathcal{A} = \det
\begin{bmatrix}
\overline{r}_1 & \overline{r}_1' & \cdots & \overline{r}_1^{(n_1 - 1)} & \cdots & \overline{r}_m & \overline{r}_m' & \cdots 
& \overline{r}_m^{(n_m - 1)}
\end{bmatrix}.
$$
(We use the transpose of $\mathcal{A}$ only to save some space.) 
Let $i \in \{ 1, \ldots, m \}$. We start by considering the determinant $\det \mathcal{A}$ as a polynomial in $x_i$ with other variables as coefficients, i.e.
$\det \mathcal{A} =: P (x_i) \in \mathbf{D}[x_1, \ldots, \hat{x}_i, \ldots, x_m][x_i]$. 
In the following we use the shorthand notation $\mathrm{D}_i := \frac{\mathrm{d}}{\mathrm{d}x_i}$.

According to the general Leibniz rule for derivatives,
\begin{multline}\label{Leibnizderiv}
P^{(k)}(x_i)=\mathrm{D}_i^k \det \mathcal{A} = 
\sum_{k_{1,1}+\ldots+k_{1,n_1}+k_{2,1}+\ldots+k_{m,n_m}=k} \binom{k}{k_{1,1}, \ldots, k_{m,n_m}} \cdot \\
\det
\begin{bmatrix}
\mathrm{D}_i^{k_{1,1}} \overline{r}_1 & \mathrm{D}_i^{k_{1,2}} \overline{r}_1' & \cdots & \mathrm{D}_i^{k_{1,n_1}} \overline{r}_1^{(n_1 - 1)} 
& \cdots & \mathrm{D}_i^{k_{m,1}} \overline{r}_m & \mathrm{D}_i^{k_{m,2}} \overline{r}_m' & \cdots 
& \mathrm{D}_i^{k_{m,n_m}} \overline{r}_m^{(n_m - 1)}
\end{bmatrix} \\
= \sum_{k_{i,1}+k_{i,2}+\ldots+k_{i,n_i}=k} \binom{k}{0,\ldots,0,k_{i,1}, \ldots, k_{i,n_i}, 0, \ldots, 0} 
\cdot \Delta(k_{i,1},k_{i,2},\ldots,k_{i,n_i})(x_i),
\end{multline}
where
$$
\Delta(k_{i,1},k_{i,2},\ldots,k_{i,n_i})(x_i):= \det
\begin{bmatrix}
* & \overline{r}^{(k_{i,1})}(x_i) & \overline{r}^{(k_{i,2}+1)}(x_i) & \ldots &  \overline{r}^{(k_{i,n_i}+n_i-1)}(x_i)& * 
\end{bmatrix}. 
$$
(The star symbol $(*)$ denotes the rest of the blocks which are unchanged.)
Now we evaluate the derivative $P^{(k)}$ at $x_j$, $j \neq i$.
First we see that
$$
\Delta(k_{i,1},k_{i,2},\ldots,k_{i,n_i})(x_j)=0
$$
if
\begin{equation*}
\begin{cases}
k_{i,s}+s-1=k_{i,t}+t-1,\quad s\ne t,\quad s,t\in\{1,\ldots,n_i\};\\
\text{or}\\
\text{there exists an}\ s\in\{1,\ldots,n_i\}\ \text{such that}\  k_{i,s}+s-1\le n_j-1.
\end{cases}
\end{equation*}
Therefore, supposing 
\begin{equation*}
\Delta(k_{i,1},k_{i,2},\ldots,k_{i,n_i})(x_j)\ne 0
\end{equation*}
implies that the numbers $k_{i,1},k_{i,2}+1,\ldots,k_{i,n_i}+n_i-1$ are distinct and strictly greater than $n_j-1$.
Hence
\begin{equation*}
\sum_{s=1}^{n_i} (k_{i,s}+s-1) \ge \sum_{t=0}^{n_i-1} (n_j+t),
\end{equation*}
which gives
\begin{equation*}
k=\sum_{s=1}^{n_i} k_{i,s} \ge \sum_{t=0}^{n_i-1} n_j = n_i n_j.
\end{equation*}

Thus
\begin{equation*}\label{derivaattaehto1}
P^{(k)} (x_j) = 0, \quad k= 0, 1, \ldots, n_i n_j -1, \quad j \in \{1, \ldots, m \} \setminus \{i\}.
\end{equation*}
Applying Lemma \ref{derivaattalemma} with $\mathbf{I} = \mathbf{D} [x_1, \ldots, \hat{x}_i,\ldots, x_m]$, we get
\begin{equation*}
(x_i - x_j)^{n_i n_j} \underset{\mathbf{D} [x_1, \ldots, \hat{x}_i,\ldots, x_m][x_i]}{\bigg|} P (x_i), \quad  j \in \{1, \ldots, m \} \setminus \{i\}.
\end{equation*}
Noting that the elements $x_g-x_h$ and $x_i-x_j$ are relatively prime when $(g,h)\ne (i,j)$, we arrive at the common factor
$$
\prod_{1  \le i < j \le m} (x_i - x_j)^{n_i n_j} \underset{\mathbf{D} [x_1, \ldots, x_m]}{\bigg|} \det \mathcal{A}.
$$
\end{proof}

\begin{corollary}\label{polyvande}
If $\deg p_i = i$, then
$$
\det \mathcal{A} = F \cdot \prod_{1  \le i < j \le m} (x_i - x_j)^{n_i n_j}
$$
for some $F \in \mathbf{D}$.
\end{corollary}

\begin{proof}
Lemma \ref{blockpolyvande} shows that
$$
\det \mathcal{A} = F \cdot \prod_{1 \le i < j \le m} (x_i - x_j)^{n_i n_j}  
$$
for some $F\in\mathbf{D}[x_1, \ldots, x_m]$.
Again, we shall consider $\det \mathcal{A}$  as a polynomial $P(x_k)$ in $x_k$ for some $k \in \{1, \ldots,m\}$.
The degrees of the entries in block $A_k$ form the following matrix:
\begin{equation*}\label{degblock}
\begin{pmatrix}
0 & 1 & 2 &\cdots & n-2 & n-1 \\
- \infty & 0 & 1 & \cdots & n-3 & n-2 \\
- \infty & - \infty & 0 & \cdots & n-4 & n-3 \\
\vdots & \vdots & \vdots & \ddots & \vdots & \vdots \\
- \infty & - \infty & - \infty & \cdots & n-n_k & n-(n_k-1) \\
- \infty & - \infty & - \infty & \cdots & n-(n_k +1) & n-n_k
\end{pmatrix}_{n_k \times n}.
\end{equation*}
%The elements in \eqref{degblock} are of the form
Above an element in position $(v,n-w)$ is of the form
$$
n-v-w, \quad v \in \{1, \ldots, n_k\}, \quad w \in \{0,1, \ldots, n-1\},
$$
where a negative number corresponds to a zero polynomial of the lower left corner.
When $\det \mathcal{A}$ is expanded using the Leibniz formula, the degree of a non-zero term is a sum
$$
\sum_{\substack{i=1\\n-i-w_i \ge 0}}^{n_k} (n-i-w_i), \quad w_i \in \{0,1, \ldots, n-1 \},
$$
where the numbers $w_i$ are pairwise distinct, so that $\sum_{i=1}^{n_k} w_i \ge \sum_{i=0}^{n_k-1} i$.
Therefore we get the upper bound
\begin{equation}\label{degPxiestimate}
\deg_{x_k} P(x_k) \le n_k n - \sum_{i=1}^{n_k} i - \sum_{i=0}^{n_k-1} i = n_k (n-n_k).
\end{equation}
On the other hand, it is easy to see that
$$
\deg_{x_k} \left( \prod_{1  \le i < j \le m} (x_i - x_j)^{n_i n_j} \right) = \sum_{\substack{j=1\\j \neq k}}^m n_k n_j = n_k (n-n_k).
$$
The degree formula and estimate \eqref{degPxiestimate} give
$$
\deg_{x_k} \det \mathcal{A} = \deg_{x_k} F + n_k (n-n_k) =\deg_{x_k} P(x_k) \le  n_k (n-n_k),
$$
implying $\deg_{x_k} F \le 0$.
Hence  
$$\det \mathcal{A} = F \cdot \prod_{1  \le i < j \le m} (x_i - x_j)^{n_i n_j},$$
where $F$ does not depend on $x_k$ with an arbitrary $k$. Thus $F \in \mathbf{D}$.
\end{proof}

Corollary \ref{polyvande} is a generalisation of the well-known polynomial Vandermonde matrix (see Krattenthaler \cite[Proposition 1]{Krattenthaler}):

\begin{corollary}\label{Kratt}
If $\deg p_i = i$, $i \in \{0, 1, \ldots, n-1\}$, and $n_j =1$ for all $j \in \{ 1, \ldots, m\}$, then
$$
\det \mathcal{A} = a_{0,0} a_{1,1} \cdots a_{n-1,n-1} \prod_{1 \le i < j \le m} (x_j - x_i),
$$
where $a_{i,i}$ are the leading coefficients of the polynomials $p_i$.
\end{corollary}

Flowe and Harris' \cite{FloweHarris} Theorem 1.1 is a special case of Corollary \ref{polyvande}, too. 
In some cases it is possible to compute the constant $F$.
This was done by Flowe and Harris \cite{FloweHarris} in their case, 
where the polynomials $p_i(x)$ were powers $x^i$ for $i=0,1, \ldots, n-1$.
Later we shall compute $F$ in a different case (see Lemma \ref{ranklemma}).

\subsection{Case B}

In Lemma \ref{blockpolyvande} we studied a block matrix in which each block starts with the same row of polynomials which are then differentiated so that the last rows of the blocks are not necessarily the same (in terms of the polynomials). 

Now we look at a matrix where rows are created by differentiating to the other direction, starting from the last row of each block. Again the first rows are the same, but last rows may differ between blocks.

Let $n_{\max} = \max_{1 \le i \le m} \{ n_i \}$. Denote
$$
B_j :=
\begin{pmatrix}
p_0^{(n_{\max} -1)} (x_j) & p_1^{(n_{\max} -1)} (x_j) & \cdots & p_{n-1}^{(n_{\max} -1)} (x_j) \\
p_0^{(n_{\max} -2)} (x_j) & p_1^{(n_{\max} -2)} (x_j) & \cdots & p_{n-1}^{(n_{\max} -2)} (x_j) \\
\vdots & \vdots & \ddots & \vdots \\
p_0^{(n_{\max}-n_j)} (x_j) & p_1^{(n_{\max}-n_j)} (x_j) & \cdots & p_{n-1}^{(n_{\max}-n_j)} (x_j) 
\end{pmatrix}_{n_j \times n},
\quad j=1, \ldots, m,
$$
and let
\begin{equation*}
\mathcal{B} :=
\begin{pmatrix}
B_1 \\
B_2 \\
\vdots \\
B_m
\end{pmatrix}_{n \times n}.
\end{equation*}

\begin{lemma}\label{blockpolyvande2}
With the above notation, we have
$$
\prod_{1 \le i < j \le m} (x_i - x_j)^{\min \left\{ n_i^2,\, n_j^2 \right\}} \underset{\mathbf{D} [x_1, \ldots, x_m]}{\bigg|} \det \mathcal{B} .
$$
\end{lemma}

\begin{proof}
With the notation in \eqref{rivi}, block $B_j$ becomes
$$
B_j^T =
\begin{bmatrix}
\overline{r}_j^{(n_{\max}-1)} & \overline{r}_j^{(n_{\max}-2)} & \cdots & \overline{r}_j^{(n_{\max}-n_j)}
\end{bmatrix}.
$$
Let $i \in \{ 1, \ldots, m \}$ and $j \in \{ 1, \ldots, m \} \setminus \{i\}$. Without loss of generality we may assume that $n_i \le n_j$. Consider the determinant $\det \mathcal{B}$ as a polynomial $P$ in $x_i$: $\det \mathcal{B}  =: P(x_i)$.

Again we use the general Leibniz rule for derivatives.
Differentiating $P$ means we are differentiating block $B_i$, since all the other blocks are constants with respect to 
$x_i$ (compare to \eqref{Leibnizderiv}):
\begin{multline*}\label{Leibnizderiv2}
\mathrm{D}_i^k \det \mathcal{B} = 
\sum_{k_{i,1}+k_{i,2}+\ldots+k_{i,n_i}=k} \binom{k}{0,\ldots,0,k_{i,1}, \ldots, k_{i,n_i}, 0, \ldots, 0} \cdot \\
\det
\begin{bmatrix}
* & \mathrm{D}_i^{k_{i,1}} \overline{r}_i^{(n_{\max}-1)} & \mathrm{D}_i^{k_{i,2}} \overline{r}_i^{(n_{\max}-2)} & \cdots 
& \mathrm{D}_i^{k_{i,n_i}} \overline{r}_i^{(n_{\max} - n_i)} & *
\end{bmatrix}\\
= \sum_{k_{i,1}+k_{i,2}+\ldots+k_{i,n_i}=k} \binom{k}{0,\ldots,0,k_{i,1}, \ldots, k_{i,n_i}, 0, \ldots, 0} 
\cdot \Delta(k_{i,1},k_{i,2},\ldots,k_{i,n_i})(x_i),
\end{multline*}
where
$$
\Delta(k_{i,1},k_{i,2},\ldots,k_{i,n_i})(x_i):= \det
\begin{bmatrix}
* & \overline{r}^{(k_{i,1}+n_{\max}-1)}(x_i) & \ldots &  \overline{r}^{(k_{i,n_i}+n_{\max} - n_i)}(x_i) & * 
\end{bmatrix}, 
$$
and the stars are again used to denote the other, unchanged blocks. 

We evaluate the derivative $P^{(k)}$ at $x_j$.
First we see that
$$
\Delta(k_{i,1},k_{i,2},\ldots,k_{i,n_i})(x_j)=0
$$
if
$$
\begin{cases}
k_{i,s}+n_{\max}-s = k_{i,t}+n_{\max}-t,\quad s\ne t,\quad s,t\in\{1,\ldots,n_i\};\\
\text{or}\\
\text{there exists an}\ s\in\{1,\ldots,n_i\}\ \text{such that}\  k_{i,s}+n_{\max}-s \in[n_{\max}-n_j,n_{\max}-1].
\end{cases}
$$
Therefore, supposing 
\begin{equation*}
\Delta(k_{i,1},k_{i,2},\ldots,k_{i,n_i})(x_j)\ne 0
\end{equation*}
implies that the numbers $k_{i,1}+n_{\max}-1,\ldots,k_{i,n_i}+n_{\max}-n_i$ are distinct and do not belong into the interval $[n_{\max}-n_j,n_{\max}-1]$.
However, the case $k_{i,s}+n_{\max}-s<n_{\max}-n_j$ never happens because $n_i\le n_j$ by our assumption.

Hence
\begin{equation*}
\sum_{s=1}^{n_i} (k_{i,s}+n_{\max}-s) \ge \sum_{t=0}^{n_i-1} (n_{\max}+t),
\end{equation*}
which gives
\begin{equation*}
k=\sum_{s=1}^{n_i} k_{i,s} \ge n_i^2 = \min \left\{n_i^2, n_j^2 \right\}.
\end{equation*}
Hence
\begin{equation*}\label{derivaattaehto3}
P^{(k)} (x_j) =0, \quad k = 0, 1, \ldots, \min \left\{ n_i^2, n_j^2 \right\}-1.
\end{equation*}
Applying Lemma \ref{derivaattalemma} with $\mathbf{I} = \mathbf{D} [x_1, \ldots, \hat{x}_i, \ldots, x_m]$, we get
\begin{equation*}
(x_i - x_j)^{\min \left\{ n_i^2, n_j^2 \right\}} \underset{\mathbf{D} [x_1, \ldots, \hat{x}_i,\ldots, x_m][x_i]}{\bigg|} P (x_i),
\end{equation*}
which proves the claim since the elements $x_g-x_h$ and $x_i-x_j$ are relatively prime when $(g,h)\ne (i,j)$.
\end{proof}

\section{Hermite-Pad\'e approximations to the exponential function: tame case}\label{sec:tame}

The problem of finding explicit Hermite-Pad\'e approximations in the case where the degrees of the polynomials are 
free parameters was resolved already by Hermite \cite{Hermite}.

Let $\alpha_1, \ldots, \alpha_m$ be distinct variables and denote $\overline{\alpha} = (\alpha_1, \ldots, \alpha_m)^T$. 
For given $l_0,l_1, \ldots,  l_m\in \mathbb{Z}_{\ge 1}$ we define 
$\sigma_{i}=\sigma_{i}\! \left(\overline l,\overline{\alpha}\right)$ by the polynomial identity
\begin{equation*}\label{}
w^{l_0}\prod_{j=1}^{m}(\alpha_j-w)^{l_j}=
\sum_{i=l_0}^{L_0}\sigma_{i}\! \left(\overline l,\overline{\alpha}\right) w^i,\quad L_0:=l_0+l_1+\ldots+l_m.
\end{equation*}
Note also the representation 
\begin{equation*}\label{}
\sigma_{i}\! \left(\overline l,\overline{\alpha}\right)=(-1)^{i-l_0}\sum_{l_0+i_1+\ldots+i_m=i}
\binom{l_1}{i_1} \cdots \binom{l_m}{i_m}\cdot \alpha_1^{l_1-i_1}\cdots\alpha_m^{l_m-i_m}.
\end{equation*}

\begin{theorem}\label{Known}
Let $l_0,l_1, \ldots,  l_m$ be positive integers and let $\alpha_1, \ldots, \alpha_m$ be distinct variables. 
Denote $\overline{\alpha} = (\alpha_1, \ldots, \alpha_m)^T$, $\overline{l} = (l_0, l_1, \ldots,  l_m)^T$, 
$L_0:=l_0+l_1+\ldots+l_m$, $L:=l_1+\ldots+l_m$. 
Then there exist polynomials $A_{\overline{l},j}(t, \overline{\alpha}) \in \Q[t, \overline{\alpha}]$ and remainders $R_{\overline{l},j}(t,\overline\alpha)$ such that
\begin{equation}\label{Padeapprox}
A_{\overline{l}, 0}(t,\overline\alpha) e^{\alpha_j t}  - A_{\overline{l}, j}(t,\overline\alpha) = R_{\overline{l},j}(t,\overline\alpha),
\quad j=1,\ldots,m,
\end{equation}
where
\begin{equation*}
\begin{cases}
\deg_t A_{\overline{l},j}(t,\overline\alpha) \le L_0-l_j,\quad j=0,1,\ldots,m;\\
L_0+1\le \underset{t=0}{\ord} \, R_{\overline{l},j}(t,\overline\alpha) < \infty,\quad j=1,\ldots,m.
\end{cases}
\end{equation*}
Moreover, the polynomial $A_{\overline{l},0}(t,\overline\alpha)$ has an explicit expression
\begin{equation}\label{Anollapolynomi}
A_{\overline{l},0}(t,\overline\alpha) = \sum_{h=0}^L (L_0-h)! \sigma_{L_0-h}\! \left(\overline l,\overline{\alpha}\right) t^h.
\end{equation}
\end{theorem}

\subsection{A new proof of Theorem \ref{Known}}

Write $A_{\overline{l}, 0}(t,\overline\alpha) =: \sum_{h=0}^{L} b_ht^h$ and
\begin{equation*}\label{sarjojenkerto}
e^{\alpha_j t} A_{\overline{l}, 0}(t,\overline\alpha) =: \sum_{N=0}^\infty r_{N,j}t^N.
\end{equation*}
Set $r_{L_0-k_j, j} = 0$ for $k_j =0, \ldots, l_j-1$, $j = 1, \ldots, m$. Then also
\begin{equation}\label{kerroinyhtalot}
\sum_{h+n=L_0-k_j} \frac{\alpha_j^n}{n!} b_h
= \sum_{h=0}^{\min\{L_0-k_j,L\}} \frac{\alpha_j^{L_0-k_j-h}}{(L_0-k_j-h)!} b_h =0
\end{equation}
for all $k_j =0, \ldots, l_j-1$, $j = 1, \ldots, m$, meaning that we have $L$ equations in $L+1$ unknowns $b_h$, $h=0,1,\dots,L$.
The equations in \eqref{kerroinyhtalot} can be written in matrix form:
$$
\mathcal{U} \overline{b} = \overline{0}, \quad \overline{b} := (b_0, b_1, \ldots, b_L)^T,
$$
with
\begin{equation*}\label{}
\mathcal{U} :=
\begin{pmatrix}
U_1 \\
U_2 \\
\vdots \\
U_m \\
\end{pmatrix}_{\!\!L \times (L+1)}, \quad
U_j :=
\begin{pmatrix}
\frac{\alpha_j^{L_0}}{L_0!} & \frac{\alpha_j^{L_0-1}}{(L_0-1)!} & \cdots & \frac{\alpha_j^{l_0}}{l_0!} \\
\frac{\alpha_j^{L_0-1}}{(L_0-1)!} & \frac{\alpha_j^{L_0-2}}{(L_0-2)!} & \cdots & \frac{\alpha_j^{l_0-1}}{(l_0-1)!} \\
\vdots & \vdots & \ddots & \vdots  \\
\frac{\alpha_j^{L_0-l_j+1}}{(L_0-l_j+1)!} & \frac{\alpha_j^{L_0-l_j}}{(L_0-l_j)!} & \cdots & \cdot \\
\end{pmatrix}_{\!\! l_j \times (L+1)}, \quad
j=1, \ldots, m,
\end{equation*}
where the typical element
\begin{equation}\label{typicalelement1}
\frac{\alpha_j^{L_0-k_j-h}}{(L_0-k_j-h)!}, \quad k_j =0, \ldots, l_j-1, \quad j = 1, \ldots, m, \quad h=0,1,\dots,L
\end{equation}
is zero whenever $L_0-k_j-h <0$.

In the following the notation $\mathcal{U}[j]$ is used for the minor obtained by removing the $j$th column of $\mathcal{U}$, $j=0,1, \ldots L$. 
The rightmost $L \times L$ minor of 
$\mathcal{U}\in \mathcal{M}_{L \times (L+1)} (\mathbb{Q}[\alpha_1, \ldots, \alpha_m] )$ 
is therefore denoted by $\mathcal{U}[0]$, and it will be considered as a polynomial in 
$\mathbb{Q}[\alpha_1, \ldots, \alpha_m]$. 
(Notice that, now and later, for technical simplicity the numbering of columns starts from $0$.)

\subsubsection{Factors of $\mathcal{U}[0]$}

\begin{lemma}\label{polyvande2}
With the above notation, we have
\begin{equation}\label{polfactorsU0}
\left( \prod_{i=1}^{m} \alpha_i^{l_0l_i} \right) \prod_{1 \le i < j \le m} (\alpha_i - \alpha_j)^{l_i l_j} 
\underset{\mathbb{Q}[\alpha_1, \ldots, \alpha_m]}{\bigg|} \mathcal{U}[0].
\end{equation}
\end{lemma}

\begin{proof}
We have
\begin{equation}\label{A0-minori}
\mathcal{U}[0]=
\begin{vmatrix}
\frac{\alpha_1^{L_0-1}}{(L_0-1)!} & \frac{\alpha_1^{L_0-2}}{(L_0-2)!} & \cdots & \frac{\alpha_1^{l_0}}{l_0!} \\
\frac{\alpha_1^{L_0-2}}{(L_0-2)!} & \frac{\alpha_1^{L_0-3}}{(L_0-3)!} & \cdots & \frac{\alpha_1^{l_0-1}}{(l_0-1)!} \\
\vdots & \vdots & \ddots & \vdots  \\
\frac{\alpha_1^{L_0-l_1}}{(L_0-l_1)!} & \frac{\alpha_1^{L_0-l_1-1}}{(L_0-l_1-1)!} & \cdots & \cdot \\
\vdots & \vdots & \ddots & \vdots \\
\vdots & \vdots & \ddots & \vdots \\
\frac{\alpha_m^{L_0-1}}{(L_0-1)!} & \frac{\alpha_m^{L_0-2}}{(L_0-2)!} & \cdots & \frac{\alpha_m^{l_0}}{l_0!} \\
\frac{\alpha_m^{L_0-2}}{(L_0-2)!} & \frac{\alpha_m^{L_0-3}}{(L_0-3)!} & \cdots & \frac{\alpha_m^{l_0-1}}{(l_0-1)!} \\
\vdots & \vdots & \ddots & \vdots  \\
\frac{\alpha_m^{L_0-l_m}}{(L_0-l_m)!} & \frac{\alpha_m^{L_0-l_m-1}}{(L_0-l_m-1)!} & \cdots & \cdot \\
\end{vmatrix}_{L \times L},
\end{equation}
where the $j$th block of $\mathcal{U}[0]$ looks like
\begin{equation*}\label{jthblockM0}
\begin{pmatrix}
\mathrm{D}_j^{0}\frac{\alpha_j^{L_0-1}}{(L_0-1)!} & \cdots & \mathrm{D}_j^{0}\frac{\alpha_j^{l_0}}{l_0!} \\
\mathrm{D}_j^{1}\frac{\alpha_j^{L_0-1}}{(L_0-1)!} & \cdots & \mathrm{D}_j^{1}\frac{\alpha_j^{l_0}}{l_0!} \\
\vdots & \ddots & \vdots  \\
\mathrm{D}_j^{l_j-1}\frac{\alpha_j^{L_0-1}}{(L_0-1)!} & \cdots & \mathrm{D}_j^{l_j-1}\frac{\alpha_j^{l_0}}{l_0!} \\
\end{pmatrix}_{l_j \times L},
\end{equation*}
denoting $\mathrm{D}_j=\frac{\mathrm{d}}{\mathrm{d}\alpha_j}$.
The factor
$$
\prod_{1 \le i < j \le m} (\alpha_i - \alpha_j)^{l_i l_j} \underset{\mathbb{Q}[\alpha_1, \ldots, \alpha_m]}{\bigg|} \mathcal{U}[0]
$$
follows thus directly from Lemma \ref{blockpolyvande}.

It remains to show the powers of alphas.
We can now use the $(l_1 \times l_1, \ldots, l_m \times l_m)$ minor expansion (Lemma \ref{manyblocksexpansion}). 
Any $l_j \times l_j$ minor of the $j$th block
\begin{equation*}\label{}
\begin{pmatrix}
\frac{\alpha_j^{L_0-1}}{(L_0-1)!} & \frac{\alpha_j^{L_0-2}}{(L_0-2)!} & \cdots & \frac{\alpha_j^{l_0}}{l_0!} \\
\frac{\alpha_j^{L_0-2}}{(L_0-2)!} & \frac{\alpha_j^{L_0-3}}{(L_0-3)!} & \cdots & \frac{\alpha_j^{l_0-1}}{(l_0-1)!} \\
\vdots & \vdots & \ddots & \vdots  \\
\frac{\alpha_j^{L_0-l_j}}{(L_0-l_j)!} & \frac{\alpha_j^{L_0-l_j-1}}{(L_0-l_j-1)!} & \cdots & \cdot \\
\end{pmatrix}_{l_j \times L}
\end{equation*}
contains the factor 
$$
\prod_{\substack{i=0 \\ L_0-v_i-w_{i}\ge 0}}^{l_j-1} \alpha_j^{L_0-i-w_{i}}, \quad w_{i}\in\{1,\ldots,L\},
$$
where the numbers $w_i$ are pairwise distinct, so that $\sum_{i=0}^{l_j-1} w_i \le \sum_{i=0}^{l_j-1} (L-i)$. 
Now
\begin{equation*}\label{}
\begin{split}
\sum_{\substack{i=0 \\ L_0-v_i-w_{i}\ge 0}}^{l_j-1} (L_0-i-w_{i}) &\ge l_j L_0-\sum_{i=0}^{l_j-1} i -\sum_{i=0}^{l_j-1} (L-i) \\
&\ge l_j(L_0-L) -\frac{(l_j-1)l_j}{2}+\frac{(l_j-1)l_j}{2}=l_0l_j.\\
\end{split}
\end{equation*}
Therefore we get
$$
\alpha_1^{l_0l_1} \cdots \alpha_m^{l_0l_m} \underset{\mathbb{Q}[\alpha_1, \ldots, \alpha_m]}{\bigg|}  \mathcal{U}[0].
$$
Since the polynomials $\alpha_1^{l_0l_1} \cdots \alpha_m^{l_0l_m}$ and $\prod_{1 \le i < j \le m} (\alpha_i - \alpha_j)^{l_i l_j}$ 
are relatively prime, we have proved the claim \eqref{polfactorsU0}.
\end{proof}

Next we show that the remaining factor is a constant, a rational number.

\begin{lemma}
There exists a rational number $F_m=F_m(l_0,\ldots,l_m)$ such that
\begin{equation}\label{M0-minori}
\mathcal{U}[0]= F_m\cdot \left( \prod_{i=1}^{m} \alpha_i^{l_0l_i} \right) \prod_{1 \le i < j \le m} (\alpha_i - \alpha_j)^{l_i l_j}. 
\end{equation}
\end{lemma}

\begin{proof}
By Lemma \ref{polyvande2}, we know that
\begin{equation}\label{U0factorisation}
\mathcal{U}[0] = F_m \cdot \alpha_1^{l_0l_1} \cdots \alpha_m^{l_0l_m} \prod_{1 \le i < j \le m} (\alpha_i - \alpha_j)^{l_i l_j}
\end{equation}
for some $F_m \in \mathbb{Q}[\alpha_1, \ldots, \alpha_m]$.
Let $k \in \{1,\ldots,m\}$. Just as in the proof of Corollary \ref{polyvande}, we see that when \eqref{A0-minori} is treated as a polynomial in $\alpha_k$ and expanded using the Leibniz formula, the degree of a non-zero term is a sum
$$
\sum_{\substack{i=0 \\ L_0-i-w_i \ge 0}}^{l_k-1} (L_0-i-w_i), \quad w_i \in \{1, \ldots, L\},
$$
where the numbers $w_i$ are pairwise disjoint, so that $\sum_{i=0}^{l_k-1} w_i \ge \sum_{i=1}^{l_k} i$. Hence we get the upper bound
\begin{equation*}
\deg_{\alpha_k} \mathcal{U}[0] \le l_k L_0 - \frac{(l_k-1)l_k}{2} - \frac{l_k(l_k+1)}{2} = l_k(L_0-l_k).
\end{equation*}
On the other hand,
\begin{equation*}
\deg_{\alpha_k} \mathcal{U}[0] = \deg_{\alpha_k} F_m + l_0l_k + \sum_{\substack{j=1\\j \neq k}}^m l_k l_j  = \deg_{\alpha_k} F_m + l_k(L_0-l_k)
\end{equation*}
by \eqref{U0factorisation}. Thus $\deg_{\alpha_k} F_m \le 0$, and since $k$ was arbitrary, we must have $F_m=F_m(l_0,\ldots,l_m) \in \Q$.
\end{proof}

\subsubsection{Rank}

\begin{lemma}\label{ranklemma}
We have $\rank_{\mathbb{Q}[\alpha_1, \ldots, \alpha_m]} \mathcal{U} = L$.
\end{lemma}

\begin{proof}
It is enough to show that the matrix $\mathcal{U}$ has one non-zero $L \times L$ minor, in this case the rightmost one, $\mathcal{U}[0]$.
Therefore, it remains to show that the constant $F_m$ is non-zero.

First we multiply both sides of equation \eqref{M0-minori} with the term
$$
\alpha_1^{\frac{l_1 (l_1+1)}{2}} \alpha_2^{\frac{l_2 (l_2+1)}{2}} \cdots \alpha_m^{\frac{l_m (l_m+1)}{2}}.
$$
Then, on the right-hand side of \eqref{M0-minori} the coefficient of the monomial
$$
\alpha_1^{\frac{l_1 (l_1+1)}{2} + l_0l_1 + l_1l_2 + \ldots + l_1l_m} 
\alpha_2^{\frac{l_2 (l_2+1)}{2} + l_0l_2 + l_2l_3 + \ldots + l_2l_m} \cdots 
\alpha_{m-1}^{\frac{l_{m-1} (l_{m-1}+1)}{2} + l_0l_{m-1} + l_{m-1}l_m} 
\alpha_m^{\frac{l_m (l_m+1)}{2} + l_0l_m}
$$
is $F_m$. 
On the left-hand side of \eqref{M0-minori} this monomial arises from the block diagonal when using the generalised minor expansion 
(Lemma \ref{manyblocksexpansion})
with $l_j \times l_j$ minors, $j=1, \ldots, m$, to expand $\mathcal{U}[0]$, where the first row of block $j$, $j=1, \ldots, m$, 
has been multiplied by $\alpha_j$, the second by $\alpha_j^2$ and so on, until the $l_j$th row:

\begin{equation*}
\begin{split}
&F_m \cdot \prod_{j=1}^m \alpha_j^{\frac{l_j (l_j+1)}{2}+l_0 l_j + l_j (l_{j+1}+\ldots+l_m)} \\
= \;
&\prod_{j=1}^m
\begin{vmatrix}
\frac{\alpha_j^{L_0-l_1-\ldots-l_{j-1}}}{(L_0-l_1-\ldots-l_{j-1}-1)!} & \frac{\alpha_j^{L_0-l_1-\ldots-l_{j-1}-1}}{(L_0-l_1-\ldots-l_{j-1}-2)!} & \cdots & \frac{\alpha_j^{L_0-l_1-\ldots-l_{j-1}-l_j+1}}{(L_0-l_1-\ldots-l_{j-1}-l_j)!} \\
\frac{\alpha_j^{L_0-l_1-\ldots-l_{j-1}}}{(L_0-l_1-\ldots-l_{j-1}-2)!} & \frac{\alpha_j^{L_0-l_1-\ldots-l_{j-1}-1}}{(L_0-l_1-\ldots-l_{j-1}-3)!} & \cdots & \frac{\alpha_j^{L_0-l_1-\ldots-l_{j-1}-l_j+1}}{(L_0-l_1-\ldots-l_{j-1}-l_j-1)!} \\
\vdots & \vdots & \ddots & \vdots  \\
\frac{\alpha_j^{L_0-l_1-\ldots-l_{j-1}}}{(L_0-l_1-\ldots-l_{j-1}-l_j)!} & \frac{\alpha_j^{L_0-l_1-\ldots-l_{j-1}-1}}{(L_0-l_1-\ldots-l_{j-1}-l_j-1)!} & \cdots & \cdot \\
\end{vmatrix}_{l_j \times l_j}\\
= \; &\left( \prod_{j=1}^m \alpha_j^{(L_0-l_1-\ldots-l_{j-1}) + (L_0-l_1-\ldots-l_{j-1}-1) + \ldots + (L_0-l_1-\ldots-l_{j-1}-l_j+1)} \right) \cdot \prod_{j=1}^m f_j,
\end{split}
\end{equation*}
with
\begin{equation}\label{Condetprod1}
f_j :=
\begin{vmatrix}
\frac{1}{(L_0-l_1-\ldots-l_{j-1}-1)!} & \frac{1}{(L_0-l_1-\ldots-l_{j-1}-2)!} & \cdots & \frac{1}{(L_0-l_1-\ldots-l_{j-1}-l_j)!} \\
\frac{1}{(L_0-l_1-\ldots-l_{j-1}-2)!} & \frac{1}{(L_0-l_1-\ldots-l_{j-1}-3)!} 
& \cdots & \frac{1}{(L_0-l_1-\ldots-l_{j-1}-l_j-1)!} \\
\vdots & \vdots & \ddots & \vdots  \\
\frac{1}{(L_0-l_1-\ldots-l_{j-1}-l_j)!} & \frac{1}{(L_0-l_1-\ldots-l_{j-1}-l_j-1)!} & \cdots & \cdot \\
\end{vmatrix}_{l_j \times l_j}.
\end{equation}
(Note that $l_1 + \ldots + l_{j-1} = 0$ when $j=1$.) We get an expression for $F_m$ as a product of determinants: $F_m(l_0,\ldots,l_m) = \prod_{j=1}^m f_j$, where the typical element
$$
\frac{1}{(L_0-r_j-c_j)!},
$$
$$r_j \in \{ 0,1, \ldots, l_j-1 \}, \quad c_j \in \{ l_1 + \ldots + l_{j-1}+1, \ldots, l_1 + \ldots + l_j \}, \quad j \in \{1, \ldots, m \},
$$
is zero whenever $L_0-r_j-c_j<0$ (recall \eqref{typicalelement1}). Now
\begin{multline}\label{detVj}
f_j =\frac{1}{(L_0-l_1-\ldots-l_{j-1}-1)!} \cdot \frac{1}{(L_0-l_1-\ldots-l_{j-1}-2)!} \cdots \frac{1}{(L_0-l_1-\ldots-l_{j-1}-l_j)!} \cdot\\
\begin{vmatrix}
(L_0-l_1-\ldots-l_{j-1}-1]_0 & (L_0-l_1-\ldots-l_{j-1}-1]_1 & \cdots & (L_0-l_1-\ldots-l_{j-1}-1]_{l_j-1} \\
(L_0-l_1-\ldots-l_{j-1}-2]_0 & (L_0-l_1-\ldots-l_{j-1}-2]_1 & \cdots & (L_0-l_1-\ldots-l_{j-1}-2]_{l_j-1} \\
\vdots & \vdots & \ddots & \vdots  \\
(L_0-l_1-\ldots-l_{j-1}-l_j]_0 & (L_0-l_1-\ldots-l_{j-1}-l_j]_1 & \cdots & (L_0-l_1-\ldots-l_{j-1}-l_j]_{l_j-1} \\
\end{vmatrix},
\end{multline}
using the falling factorials defined by
$$
(n]_0 := 1; \quad (n]_k := n (n-1) \cdots (n-k+1), \quad k =1, \ldots, n.
$$
This notation works for the zero elements as well, since now our typical element is
$$
(L_0-c_j]_{r_j},
$$
$$
r_j \in \{ 0,1, \ldots, l_j-1 \}, \quad c_j \in \{ l_1 + \ldots + l_{j-1}+1, \ldots, l_1 + \ldots + l_j \}, \quad j \in \{1, \ldots, m \},
$$
and if $L_0-r_j-c_j<0$, then
$$
(L_0-c_j]_{r_j} = (L_0-c_j) (L_0-c_j-1) \cdots (L_0-c_j-r_j+1) = 0.
$$
Most importantly, the falling factorial $(n]_k$ is a polynomial in $n$ of degree $k$. 
This means that the determinants in \eqref{Condetprod1} are essentially polynomial Vandermonde determinants, the polynomials being in this case $(x]_0, (x]_1, \ldots, (x]_{l_j-1}$. The evaluation of \eqref{detVj} follows from Corollary \ref{Kratt}: Let us further denote
$$
c_{j,k} :=  l_1 + \ldots + l_{j-1}+k,
$$
so that
\begin{equation}\label{Cjlaskettu}
\begin{split}
f_j &= \frac{1}{(L_0-c_{j,1})!} \cdot \frac{1}{(L_0-c_{j,2})!} \cdots \frac{1}{(L_0-c_{j,l_j})!} \cdot \prod_{1 \le h < k \le l_j} ((L_0-c_{j,k}) - (L_0-c_{j,h})) \\
%&= \frac{1}{(L_0-c_{j,1})!} \cdot \frac{1}{(L_0-c_{j,2})!} \cdots \frac{1}{(L_0-c_{j,l_j})!} \cdot \prod_{1 \le h < k \le l_j} (c_{j,h}-c_{j,k}) \\
&= \frac{1}{(L_0-c_{j,1})!} \cdot \frac{1}{(L_0-c_{j,2})!} \cdots \frac{1}{(L_0-c_{j,l_j})!} \cdot \prod_{1 \le h < k \le l_j} (h-k) \neq 0\\
\end{split}
\end{equation}
for all $j \in \{ 1, \ldots, m \}$.
\end{proof}

\subsubsection{Cramer's rule}\label{Cramersrule}

Remember that $\mathcal{U}[j]$ denotes the minor obtained by removing the $j$th column of $\mathcal{U}$, $j=0,1, \ldots L$. 
By Cramer's rule, the equation
$$
\mathcal{U} \overline{b} = \overline{0},\quad \overline{b}=(b_0,b_1,\ldots,b_L)^T, 
$$
has a solution
\begin{equation}\label{solvector}
[b_0, b_1, \ldots, b_L] = 
\left[ \mathcal{U}[0] : - \mathcal{U}[1] : \ldots : (-1)^L \mathcal{U}[L] \right],
\end{equation}
where one of the minors is appropriately non-zero, as was shown in Lemma \ref{ranklemma}.
Here we use the notation of homogeneous coordinates:
$$
[ h_1 : h_2 : \ldots : h_n ] = \left\{ t(h_1, h_2, \ldots, h_n)^T \; \big| \; t \in\mathbb{Q}[\alpha_1, \ldots, \alpha_m]\setminus \{ 0 \} \right\}
$$
for any $(h_1, h_2, \ldots, h_n)^T\in \left( \mathbb{Q}[\alpha_1, \ldots, \alpha_m] \right)^n\setminus \{ \overline{0} \}$.

To solve our Hermite-Pad\'e approximation problem, we complete the matrix $\mathcal{U} $ into an 
$(L+1) \times (L+1)$ square matrix $\mathcal{S}$ by adding on bottom of it the row
\begin{equation*}\label{newrow}
\begin{pmatrix}
\frac{x^{L_0}}{L_0!} & \frac{x^{L_0-1}}{(L_0-1)!} & \cdots & \frac{x^{l_0}}{l_0!} \\
\end{pmatrix},
\end{equation*}
where $x=\alpha_{m+1}$ is a new, technical variable. Now the Laplace expansion along the last row implies
\begin{equation}\label{rep1}
\det \mathcal{S} =(-1)^{L} \sum_{h=0}^L (-1)^h \mathcal{U}[h] \frac{x^{L_0-h}}{(L_0-h)!}.
\end{equation}

On the other hand, we know how to compute $\det \mathcal{S}$: 
just replace $L_0$ with $L_0+1$ and $m$ with $m+1$ in the proof of 
Lemma \ref{ranklemma} ($l_{m+1} =1$). It tells us that
\begin{equation*}\label{detS}
\begin{split}
\det \mathcal{S} 
&=  F_{m+1}(l_0,\ldots,l_{m+1}) \cdot  \alpha_1^{l_0l_1} \cdots \alpha_m^{l_0l_m} x^{l_0l_{m+1}} 
\prod_{1 \le i < j \le m+1} (\alpha_i - \alpha_j)^{l_i l_j} \\
&=  F_{m+1}(l_0,\ldots,l_{m+1}) \cdot \alpha_1^{l_0l_1} \cdots \alpha_m^{l_0l_m} 
\prod_{1 \le i < j \le m} (\alpha_i - \alpha_j)^{l_i l_j}\, x^{l_0} \prod_{i=1}^m (\alpha_i - x)^{l_i},
\end{split}
\end{equation*}
where $F_{m+1}(l_0,\ldots,l_{m+1})$ is given by \eqref{Condetprod1} and \eqref{Cjlaskettu} (with $L_0+1$ instead of $L_0$ and $m+1$ instead of $m$).

Now
\begin{equation*}\label{taunmaarittely}
x^{l_0} \prod_{i=1}^m (\alpha_i - x)^{l_i} = \sum_{i=l_0}^{L_0} \sigma_i\! \left(\overline l,\overline{\alpha}\right) x^i,
\end{equation*}
where
$$
\sigma_i\! \left(\overline l,\overline{\alpha}\right) = (-1)^{i-l_0} \sum_{l_0+i_1 + \ldots + i_m = i} \binom{l_1}{i_1} \cdots \binom{l_m}{i_m} 
\alpha_1^{l_1-i_1} \cdots \alpha_m^{l_m-i_m}
$$
and $l_0+l_1+ \ldots +l_m=L_0$. Thus
\begin{equation}\label{rep2}
\begin{split}
\det \mathcal{S} &= F_{m+1}(l_0,\ldots,l_{m+1}) \cdot  \alpha_1^{l_0l_1} \cdots \alpha_m^{l_0l_m} \prod_{1 \le i < j \le m} (\alpha_i - \alpha_j)^{l_i l_j} \sum_{i=l_0}^{L_0} \sigma_i \! \left(\overline l,\overline{\alpha}\right) x^i\\
&=F_{m+1}(l_0,\ldots,l_{m+1}) \cdot  \alpha_1^{l_0l_1} \cdots \alpha_m^{l_0l_m}
\prod_{1 \le i < j \le m} (\alpha_i - \alpha_j)^{l_i l_j} \sum_{h=0}^{L} \sigma_{L_0-h}\! \left(\overline l,\overline{\alpha}\right) x^{L_0-h}.
\end{split}
\end{equation}
Comparison of the two representations \eqref{rep1} and \eqref{rep2} (as polynomials in $x$) yields
\begin{equation}\label{hminorexp}
\mathcal{U}[h] = (-1)^{h}(L_0-h)! \sigma_{L_0-h}\! \left(\overline l,\overline{\alpha}\right)\cdot (-1)^{L} F_{m+1} 
\cdot  \alpha_1^{l_0l_1} \cdots \alpha_m^{l_0l_m}
\prod_{1 \le i < j \le m} (\alpha_i - \alpha_j)^{l_i l_j}
\end{equation}
for $h=0,1, \ldots, L$. Hence, \eqref{hminorexp} shows the complete factorisation of all the $L\times L$ minors of 
the matrix $\mathcal{U}$. Secondly, we see that all the above minors are non-zero polynomials in  $\mathbb{Q}[\alpha_1, \ldots, \alpha_m]$.

Because of the homogeneous coordinates in \eqref{solvector}, we get
\begin{equation*}
\begin{split}
[b_0, b_1, \ldots, b_L]
%&= \left[ \mathcal{A}[0] : - \mathcal{A}[1] : \ldots : (-1)^L \mathcal{A}[L] \right]\\ 
&=\left[ L_0! \sigma_{L_0}\! \left(\overline l,\overline{\alpha}\right) : (L_0-1)! \sigma_{L_0-1}\! \left(\overline l,\overline{\alpha}\right) : \ldots : l_0! \sigma_{l_0}\! \left(\overline l,\overline{\alpha}\right) \right].
\end{split}
\end{equation*}
Then $A_{\overline{l},0} (t,\overline\alpha)$ in \eqref{Padeapprox} becomes
$$
A_{\overline{l},0} (t,\overline\alpha) = \sum_{h=0}^L (L_0-h)! \sigma_{L_0-h}\! \left(\overline l,\overline{\alpha}\right) t^h.
$$
\begin{flushright}
\qed
\end{flushright}

Before going to the 'twin problem' we remark that the seemingly extra parameter $l_0$ is crucial in the applications of type II Hermite-Pad\'e approximations \eqref{Padeapprox} (see \cite{Hermite}, \cite{EMS2016}).

\section{Hermite-Pad\'e approximations to the exponential function: wild case}\label{sec:wild}

\subsection{The twin problem}

As explained in the Introduction, the problem of finding explicit type II Hermite-Pad\'e approximations in the case where 
the degrees of the polynomials are the same but the orders of the remainders are free parameters is yet unsolved. 
We called it the 'twin problem', stated as follows: Find an explicit polynomial $B_{\overline{l},0} (t,\overline\alpha)$, polynomials
$B_{\overline{l},j}(t, \overline{\alpha})$ and remainders $S_{\overline{l},j}(t,\overline\alpha)$, $j=1,\ldots,m$, satisfying
\begin{equation}\label{TwinHermite2}
B_{\overline{l},0} (t,\overline\alpha)  e^{\alpha_j t} - B_{\overline{l},j} (t,\overline\alpha)  =
S_{\overline{l},j} (t,\overline\alpha) , \quad j=1, \ldots, m,
\end{equation}
with $\overline{l} = (l_1, \ldots,  l_m)^T\in\mathbb{Z}_{\ge 1}^m$, $L:=l_1+\ldots+l_m$ and 
\begin{equation}\label{degreeorderwild} 
\begin{cases}
\deg_t B_{\overline{l},j} (t,\overline\alpha)  \le L, &j=0, \ldots, m; \\
L+l_j +1 \le \underset{t=0}{\ord} \, S_{\overline{l},j}(t,\overline\alpha) < \infty, &j=1, \ldots, m.
\end{cases}
\end{equation}
We note that in the setting of the twin problem it is not possible to use the parameter $L_0$ in a similar way as in the tame case.
Namely, by replacing $L$ with $L_0$ in \eqref{degreeorderwild} yields to $L=l_1 + \ldots + l_m$ equations with
$L_0+1$ unknowns. If $L_0>L$, the resulting Pad\'e polynomial $B_{\overline{l},0}(t,\overline\alpha)$ would not be unique.

The twin Pad\'e approximation \eqref{TwinHermite2} is a special case of the following, more general Pad\'e-type approximation:
\begin{equation}\label{DualHermite2}
B_{\overline{\nu},0} (t,\overline{\alpha}) e^{\alpha_j t} - B_{\overline{\nu},j} (t,\overline{\alpha}) = 
S_{\overline{\nu},j} (t,\overline{\alpha}), \quad j=1, \ldots, m,
\end{equation}
where $\overline{\nu} = ( \nu_1, \ldots, \nu_m)^T\in\mathbb{Z}_{\ge 1}^m$, $\nu_1\le l_1,\ldots,\nu_m\le l_m$,
$\nu_1 + \ldots + \nu_m =: M \le L$ and
\begin{equation*}
\begin{cases}
\deg_t B_{\overline{\nu},j} (t,\overline{\alpha}) \le L, &j=0, \ldots, m; \\
L+\nu_j +1 \le \underset{t=0}{\ord} \, S_{\overline{\nu},j} (t,\overline{\alpha})  < \infty, &j=1, \ldots, m.
\end{cases}
\end{equation*}

\subsection{Siegel's lemma}

If $M<L$, there is no unique solution $B_{\overline{\nu},0} (t,\overline{\alpha})$ to the Pad\'e-type approximation 
equations \eqref{DualHermite2}, neither is it known how to find an explicit solution.
Therefore, we now switch to integers and apply Siegel's lemma. 

Choose now $\overline{\alpha} = \overline{a} := (a_1, \ldots, a_m)^T$, where $a_1, \ldots, a_m$ are pairwise different, non-zero integers. Write
\begin{equation}\label{Bnu0-2}
B_{\overline{\nu},0}(t,\overline{a}) = \sum_{h=0}^L c_h \frac{L!}{h!} t^h,
\end{equation}
where $\overline{\nu} := (\nu_1, \ldots, \nu_m )^T$, and the numbers
$\nu_1, \ldots, \nu_m, l_1, \ldots, l_m \in \Z_{\ge 1}$ satisfy
$$
1 \le \nu_j \le l_j, \quad M := \nu_1 + \ldots + \nu_m \le L := l_1 + \ldots + l_m.
$$
Then \eqref{DualHermite2} yields the matrix equation
\begin{equation}\label{unknowneq}
\mathcal{V} \overline{c} = \overline{0}, \quad \overline{c} := (c_0, c_1, \ldots, c_L)^T,
\end{equation}
with
\begin{equation*}\label{AKTmatrixX}
\mathcal{V} = \mathcal{V}(\overline{a}) :=
\begin{pmatrix}
\binom{L + 1}{0} a_1^L & \binom{L + 1}{1} a_1^{L-1} & \cdots & \binom{L + 1}{L - 1} a_1 & \binom{L +1}{L} \\
\binom{L + 2}{0} a_1^L & \binom{L + 2}{1} a_1^{L-1} & \cdots & \binom{L + 2}{L - 1} a_1 & \binom{L +2}{L} \\
\vdots & \vdots & \ddots & \vdots & \vdots \\
\binom{L + \nu_1}{0} a_1^L & \binom{L + \nu_1}{1} a_1^{L-1} & \cdots & \binom{L + \nu_1}{L - 1} a_1 & \binom{L +\nu_1}{L} \\

\vdots & \vdots & \ddots & \vdots & \vdots \\

\vdots & \vdots & \ddots & \vdots & \vdots \\

\binom{L + 1}{0} a_m^L & \binom{L + 1}{1} a_m^{L-1} & \cdots & \binom{L + 1}{L - 1} a_m & \binom{L +1}{L} \\
\binom{L + 2}{0} a_m^L & \binom{L + 2}{1} a_m^{L-1} & \cdots & \binom{L + 2}{L - 1} a_m & \binom{L +2}{L} \\
\vdots & \vdots & \ddots & \vdots & \vdots \\
\binom{L + \nu_m}{0} a_m^L & \binom{L + \nu_m}{1} a_m^{L-1} & \cdots & \binom{L + \nu_m}{L - 1} a_m & \binom{L +\nu_m}{L} \\
\end{pmatrix}_{\!M \times (L+1)},
\end{equation*}
for which Siegel's lemma produces a non-zero integer solution with a non-trivial upper bound (see \eqref{SiegelboundTT} below).
The term $\frac{L!}{h!}$ has been included in \eqref{Bnu0-2} just to get an integer matrix in \eqref{unknowneq} and thus make the use of Siegel's lemma possible.
A standard application of Siegel's lemma is the following Lemma \ref{AKTlemma41}, a particular case of \cite[Lemma 4.1]{AKT}.

\begin{lemma}\label{AKTlemma41}
There exist a non-zero polynomial
\begin{equation*}\label{A0pol}
B_{\overline{\nu},0}(t,\overline{a}) = \sum_{h=0}^L c_h \frac{L!}{h!} t^h \in \Z[t]
\end{equation*}
and non-zero polynomials 
$B_{\overline{\nu},j} (t,\overline{a})  \in \Z [t, \overline{a}]$, $j = 1, \ldots, m$, 
such that
\begin{equation*}\label{Padeapproaunknown}
B_{\overline{\nu},0} (t,\overline{a})  e^{a_j t} - B_{\overline{\nu},j} (t,\overline{a})  = 
S_{\overline{\nu},j} (t,\overline{a}) , \quad j=1, \ldots, m,
\end{equation*}
where
\begin{equation}\label{cond}
\begin{cases}
\deg_t B_{\overline{\nu}, j} (t,\overline{a}) \le L, &j=0, \ldots, m; \\
L+\nu_j +1 \le \underset{t=0}{\ord} \, S_{\overline{\nu},j} (t,\overline{a}) < \infty, &j=1, \ldots, m.
\end{cases}
\end{equation}
Moreover, we have $(c_0, c_1, \ldots, c_L)^T \in \Z^{L+1} \setminus \{ \bar{0} \}$ and
\begin{equation}\label{SiegelboundTT}
| c_h | \le  \left( f^{ML} g^{\frac{M^2}{2}} \right)^{\frac{1}{L+1-M}}, \quad h=0,1, \ldots, L,
\end{equation}
where
$$
f = f(\overline{a}) = \max_{1 \le j \le m} \{ | a_j | + 1 \}, 
\quad g = g(\overline{a}) = \max_{1 \le j \le m} \left\{ 1 + \frac{1}{| a_j |} \right\}.
$$
\end{lemma}

\begin{proof} 
For the sake of completeness, we shall reproduce the proof from \cite{AKT}.
Let
\begin{equation}\label{B0eap}
B_{\overline{\nu},0} (t,\overline{a})  e^{a_j t} = \sum_{N=0}^\infty r_{N, j} t^N, \quad j = 1, \ldots, m,
\end{equation}
where
$$
r_{N, j} := \sum_{\substack{h+n=N \\ 0 \leq h \leq L}} c_h \frac{L!}{h!} \frac{a_j^n}{n!}.
$$
Cut the series \eqref{B0eap} after $L + 1$ terms and let
$$
B_{\overline{\nu},j}(t,\overline{a}) := \sum_{N=0}^L r_{N, j} t^N, \quad j = 1, \ldots, m.
$$
Set $r_{L+i_j, j} = 0$ for $i_j = 1, \ldots, \nu_j$, $j = 1, \ldots, m$. Then also
\begin{equation}\label{ajyjch}
0 = \frac{(L + i_j)!}{L!} \frac{1}{a_j^{i_j}} r_{L+i_j, j}
%= \frac{(L + i_j)!}{L!} \frac{1}{a_j^{i_j}} \sum_{\substack{h+n=L+i_j \\ 0 \leq h \leq L}} c_h \frac{L!}{h! n!} a_j^n\\
= \sum_{\substack{h+n=L+i_j \\ 0 \leq h \leq L}} \frac{(L + i_j)!}{h! n!} a_j^{n-i_j}  c_h
%= \sum_{h=0}^L \frac{(L + i_j)!}{h! (L + i_j - h)!} a_j^{L-h}  c_h\\
= \sum_{h=0}^L \binom{L + i_j}{h} a_j^{L-h}  c_h
\end{equation}
for $i_j = 1, \ldots, \nu_j$, $j = 1, \ldots, m$, meaning that we have $M$ equations in $L + 1$ unknowns $c_h, \:h= 0, 1, \ldots, L$, 
with coefficients $\binom{L + i_j}{h} a_j^{L-h}  \in \Z$. These coefficients satisfy the estimate
\begin{align*}
A_{j, i_j} &:= \sum_{h=0}^L \left| \binom{L + i_j}{h} a_j^{L-h}  \right|
%= \frac{1}{| a_j |^{i_j}} \sum_{h=0}^L \binom{L + i_j}{h} \left| a_j \right|^{L+i_j-h} \\
< \frac{1}{| a_j |^{i_j}} \sum_{h=0}^{L+i_j} \binom{L + i_j}{h} \left| a_j \right|^{L+i_j-h} \\
&= \frac{1}{| a_j |^{i_j}} \left( | a_j | + 1 \right)^{L+i_j}
= \left( | a_j | + 1 \right)^L \left( 1 + \frac{1}{| a_j |} \right)^{i_j},
\end{align*}
when  $i_j = 1, \ldots, \nu_j$, $j = 1, \ldots, m$. (Recall that in order to apply Lemma \ref{Thue-Siegel's}, 
the product of the row sums \eqref{rivisumma} is needed.) Then
\begin{align*}
\prod_{j, i_j} A_{j, i_j} &< \left( \prod_{j, i_j} \left( | a_j | + 1 \right)^L \right) \prod_{j, i_j} \left( 1 + \frac{1}{| a_j |} \right)^{i_j} \\
&= \left( \prod_j \left( | a_j | + 1 \right)^{\nu_j L} \right) \prod_j \left( 1 + \frac{1}{| a_j |} \right)^{\frac{\nu_j (\nu_j + 1)}{2}}
\leq f^{ML} g^{\frac{M^2}{2}}.
\end{align*}
The last step follows by employing the definition $M = \nu_1 + \ldots + \nu_m$, which implies 
$\sum_{j=1}^m \nu_j (\nu_j + 1) = \sum_{j=1}^m \nu_j^2 + \sum_{j=1}^m \nu_j \leq M^2$. 
Thus by Lemma \ref{Thue-Siegel's} there exists a solution 
$\overline{c} =(c_0, c_1, \ldots, c_L)^T \in \Z^{L+1} \setminus \{ \bar{0} \}$ to the group of $M$ equations derived in \eqref{ajyjch} with
$$
| c_h | \leq \left( f^{ML} g^{\frac{M^2}{2}} \right)^{\frac{1}{L+1-M}} , \quad h = 0, 1, \ldots, L.
$$

Writing
$$
S_{\overline{\nu},j} (t,\overline{a})  := \sum_{N = L + \nu_j +1}^\infty r_{N, j} t^N
$$
we get
$$
B_{\overline{\nu},0}(t,\overline{a}) e^{a_j t} - B_{\overline{\nu},j} (t,\overline{a})  = 
S_{\overline{\nu},j} (t,\overline{a}) , \quad j = 1, \ldots, m.
$$
Here $B_{\overline{\nu},j} (t,\overline{a})$ are non-zero polynomials for all $j = 0, 1, \ldots, m$, 
since the solution $\overline{c}$ is a non-zero vector. Conditions \eqref{cond} are also satisfied, as the series $S_{\overline{\nu},j} (t)$ is non-zero too.
\end{proof}

\subsection{The Bombieri-Vaaler version of Siegel's lemma}

Next we are going to examine how to improve the estimate in \eqref{SiegelboundTT} by using the Bombieri-Vaaler version of 
Siegel's lemma (Lemma \ref{Bombieri-Vaaler}).
Assuming that the rank of $\mathcal{V}(\overline{a})$ over $\mathbb{Z}$ is $M$, 
it follows from Lemma \ref{Bombieri-Vaaler} that equation \eqref{unknowneq} has a solution 
$\overline{c}=(c_0,\ldots,c_L)^T \in \Z^{L+1} \setminus \{ \overline{0} \}$ with
$$
\left\| \overline{c} \right\|_\infty:= \max_{0 \le k \le L} |c_{k}|
 \le \left( \frac{\sqrt{\det (\mathcal{V}(\overline{a})\mathcal{V}(\overline{a})^T )}}{D(\overline{a})} \right)^{\frac{1}{N-M}}
\le \left( \frac{\prod\limits_{m=1}^{M}\|\underline{w}_m\|_{1}}{D(\overline{a})}\right)^{\frac{1}{N-M}},
$$
where $D(\overline{a})$ is the greatest common divisor of all the $M \times M$ minors of $\mathcal{V}(\overline{a})$,
$\underline{w}_m$ denotes the $m$th row of the matrix $\mathcal{V}(\overline{a})$ and $\left\|(v_1,\ldots,v_N)^T \right\|_{1}:=|v_1|+\ldots+|v_N|$.

Therefore, the bound in \eqref{SiegelboundTT} can be improved if
we can find a relatively big common factor from the $M \times M$ minors of the matrix $\mathcal{V}(\overline{a})$.
Such a factor indeed exists, as is stated in Theorem \ref{IntegerFactortheorem}.
Theorem \ref{IntegerFactortheorem} is a direct corollary of Theorem \ref{polynomialfactortheorem} which will be presented and proved in the following section. 

\subsection{Common factor}

To prove Theorem \ref{IntegerFactortheorem} for an arbitrary $m$-tuple of integers, we need to treat the determinants as polynomials. 
Therefore we use our original variables $\alpha_i$ and consider the matrix 
$\mathcal{V}(\overline{\alpha}) \in \mathcal{M}_{M \times (L+1)} (\mathbb{Z}[\alpha_1, \ldots, \alpha_m])$.

Let $\widehat{D}(\overline{\alpha}) \in \Z[\alpha_1, \ldots, \alpha_m]$ be the greatest common divisor of the $M \times M$ minors of the matrix $\mathcal{V}(\overline{\alpha})$ (we assume that $\widehat{D}(\overline{\alpha})$ is a primitive polynomial). 
Finding $\widehat{D}(\overline{\alpha})$ means that we are interested in the divisors of an arbitrary $M \times M$ minor of $\mathcal{V}(\overline{\alpha})$, 
denoted by
\begin{equation*}\label{minorAKT}
\det \mathcal{W} :=
\begin{vmatrix}
\binom{L + 1}{e_1} \alpha_1^{L-e_1} & \binom{L + 1}{e_2} \alpha_1^{L-e_2} & \cdots & \binom{L + 1}{e_M} \alpha_1^{L-e_M}\\
\vdots & \vdots & \ddots & \vdots \\
\binom{L + \nu_1}{e_1} \alpha_1^{L-e_1} & \binom{L + \nu_1}{e_2} \alpha_1^{L-e_2} & \cdots & \binom{L + \nu_1}{e_M} \alpha_1^{L-e_M}\\

\vdots & \vdots & \ddots & \vdots \\
\vdots & \vdots & \ddots & \vdots \\

\binom{L + 1}{e_1} \alpha_m^{L-e_1} & \binom{L + 1}{e_2} \alpha_m^{L-e_2} & \cdots & \binom{L + 1}{e_M} \alpha_m^{L-e_M} \\
\vdots & \vdots & \ddots & \vdots \\
\binom{L + \nu_m}{e_1} \alpha_m^{L-e_1} & \binom{L + \nu_m}{e_2} \alpha_m^{L-e_2} & \cdots & \binom{L + \nu_m}{e_M} \alpha_m^{L-e_M} \\
\end{vmatrix}
= \det
\begin{pmatrix}
W_1 \\
W_2 \\
\vdots \\
W_m
\end{pmatrix},
\end{equation*}
where $0 \le e_1 < e_2 < \ldots < e_M \le L$ and
\begin{equation}\label{blockMj}
W_j :=
\begin{pmatrix}
\binom{L + 1}{e_1} \alpha_j^{L-e_1} & \binom{L + 1}{e_2} \alpha_j^{L-e_2} & \cdots & \binom{L + 1}{e_M} \alpha_j^{L-e_M}\\
\vdots & \vdots & \ddots & \vdots \\
\binom{L + \nu_j}{e_1} \alpha_j^{L-e_1} & \binom{L + \nu_j}{e_2} \alpha_j^{L-e_2} & \cdots & \binom{L + \nu_j}{e_M} \alpha_j^{L-e_M}
\end{pmatrix}_{\!\nu_j \times M},
\quad j=1, \ldots, m,
\end{equation}
is the $j$th block of the matrix $\mathcal{W}$ defining our arbitrary minor. Clearly $\det \mathcal{W} \in \Z[\alpha_1, \ldots, \alpha_m]$.

\begin{theorem}\label{polynomialfactortheorem}
We have
\begin{equation}\label{polynomialfactor}
\left( \prod_{1 \le  j \le m} \alpha_j^{\binom{\nu_j}{2}} \right)
\prod_{1 \le i < j \le m} (\alpha_i - \alpha_j)^{\min \left\{ \nu_i^2,\, \nu_j^2 \right\}} 
\underset{\mathbb{Z}[\alpha_1, \ldots, \alpha_m]}{\bigg|} \widehat{D}(\overline{\alpha}).
\end{equation}
\end{theorem}

\begin{proof}
The factor $\alpha_1^{\binom{\nu_1}{2}} \cdots \alpha_m^{\binom{\nu_m}{2}}$ follows by using 
the generalised minor expansion (Lemma \ref{manyblocksexpansion}) with $\nu_j \times \nu_j$, $j=1, \ldots, m$, minors. 
According to Lemma \ref{manyblocksexpansion}, then $\det \mathcal{W}$ is a sum of products of $\nu_j \times \nu_j$ minors, 
each taken from block $W_j$. Looking at \eqref{blockMj}, we see that any $\nu_j \times \nu_j$ minor of $W_j$ contains at least the factor
$$
\alpha_j^{(L-e_M) + (L-e_{M-1}) + \ldots + (L-e_{M-\nu_j+1})} \ge \alpha_j^{0+1+\ldots+\nu_j-1} = \alpha_j^{\binom{\nu_j}{2}}.
$$
Thus
$$
\prod_{1 \le  j \le m} \alpha_j^{\binom{\nu_j}{2}} \underset{\Z [\alpha_1, \ldots, \alpha_m]}{\bigg|} \det \mathcal{W}.
$$

It remains to show the Vandermonde-type factor. If each row $(j, i_j)$ of $\det \mathcal{W}$ is multiplied by $\frac{\alpha_j^{i_j}}{(L+i_j)!}$ 
and column $k$ by $e_k!$, we get a new determinant
\begin{align*}
\det \mathcal{\widehat{W}} &:= \left( \left( \prod_{j=1}^m \prod_{i_j = 1}^{\nu_j} \frac{\alpha_j^{i_j}}{(L+i_j)!} \right) \prod_{k=1}^M e_k! \right) \det \mathcal{W} \\
&=
\begin{vmatrix}
\frac{\alpha_1^{L+1-e_1}}{(L+1 -e_1)!} & \frac{\alpha_1^{L+1-e_2}}{(L+1-e_2)!} & \cdots & \frac{\alpha_1^{L+1-e_M}}{(L+1-e_M)!}\\
\vdots & \vdots & \ddots & \vdots \\
\vdots & \vdots & \ddots & \vdots \\
\frac{\alpha_m^{L+\nu_m-e_1}}{(L+\nu_m -e_1)!} & \frac{\alpha_m^{L+\nu_m-e_2}}{(L+\nu_m-e_2)!} & \cdots & \frac{\alpha_m^{L+\nu_m-e_M}}{(L+\nu_m-e_M)!}\\
\end{vmatrix}.
\end{align*}
The $j$th block of $\mathcal{\widehat{W}}$ looks like
\begin{align*}
\widehat{W}_j &:=
\begin{pmatrix}
\frac{\alpha_j^{L+1-e_1}}{(L+1 -e_1)!} & \frac{\alpha_j^{L+1-e_2}}{(L+1-e_2)!} & \cdots & \frac{\alpha_j^{L+1-e_M}}{(L+1-e_M)!}\\
\frac{\alpha_j^{L+2-e_1}}{(L+2 -e_1)!} & \frac{\alpha_j^{L+2-e_2}}{(L+2-e_2)!} & \cdots & \frac{\alpha_j^{L+2-e_M}}{(L+2-e_M)!}\\
\vdots & \vdots & \ddots & \vdots \\
\frac{\alpha_j^{L+\nu_j-e_1}}{(L+\nu_j -e_1)!} & \frac{\alpha_j^{L+\nu_j-e_2}}{(L+\nu_j-e_2)!} & \cdots 
& \frac{\alpha_j^{L+\nu_j-e_M}}{(L+\nu_j-e_M)!}\\
\end{pmatrix}_{\nu_j \times M}. \\
\end{align*}
Notice that upper rows are derivatives of the last row, whence we may apply Lemma \ref{blockpolyvande2}, arriving at the common factor
\begin{equation}\label{jakaadetM}
\prod_{1  \le i < j \le m} (\alpha_i - \alpha_j)^{\min \left\{ \nu_i^2, \nu_j^2 \right\}} \underset{\Q [\alpha_1, \ldots, \alpha_m]}{\bigg|} \det \mathcal{\widehat{W}}.
\end{equation}

We have the connection
\begin{equation}\label{connection1}
\det \mathcal{\widehat{W}} = \left( \left( \prod_{j=1}^m \prod_{i_j = 1}^{\nu_j} \frac{\alpha_j^{i_j}}{(L+i_j)!} \right)
\prod_{k=1}^M e_k! \right) \det \mathcal{W},
\end{equation}
so
\begin{equation}\label{connection}
\left( \left( \prod_{j=1}^m \prod_{i_j = 1}^{\nu_j} (L+i_j)! \right) \prod_{k=1}^M \frac{1}{e_k!} \right) \det \mathcal{\widehat{W}}
= \left( \prod_{j=1}^m \prod_{i_j = 1}^{\nu_j} \alpha_j^{i_j} \right) \det \mathcal{W} \in \Z[\alpha_1, \ldots, \alpha_m].
\end{equation}
Since the polynomials
$$
\prod_{j=1}^m \prod_{i_j = 1}^{\nu_j} \alpha_j^{i_j} = \prod_{j=1}^m \alpha_j^{\binom{\nu_j+1}{2}} \quad \text{and} \quad \prod_{1  \le i < j \le m} (\alpha_i - \alpha_j)^{\min \left\{ \nu_i^2, \nu_j^2 \right\}}
$$
are relatively prime, property \eqref{jakaadetM} and equation \eqref{connection} now imply that
$$
\prod_{1  \le i < j \le m} (\alpha_i - \alpha_j)^{\min \left\{ \nu_i^2, \nu_j^2 \right\}} 
\underset{\Q [\alpha_1, \ldots, \alpha_m]}{\bigg|} \det \mathcal{W}.
$$
Hence
$$
\det \mathcal{W} = Q \cdot \left( \prod_{1 \le  j \le m} \alpha_j^{\binom{\nu_j}{2}} \right)
\prod_{1  \le i < j \le m} (\alpha_i - \alpha_j)^{\min \left\{ \nu_i^2, \nu_j^2 \right\}}
$$
for some $Q = Q(\alpha_1, \ldots, \alpha_m) \in \Q[\alpha_1, \ldots, \alpha_m]$.
But since $\det \mathcal{W} \in \Z[\alpha_1, \ldots, \alpha_m]$ and the polynomials 
$\prod_{1 \le  j \le m} \alpha_j^{\binom{\nu_j}{2}}$,
$\prod_{1  \le i < j \le m} (\alpha_i - \alpha_j)^{\min \left\{ \nu_i^2, \nu_j^2 \right\}}\in \Z[\alpha_1, \ldots, \alpha_m]$ 
are primitive, we have $Q \in \Z[\alpha_1, \ldots, \alpha_m]$. Hence
$$
\prod_{1  \le i < j \le m} (\alpha_i - \alpha_j)^{\min \left\{ \nu_i^2,\, \nu_j^2 \right\}} 
\underset{\Z [\alpha_1, \ldots, \alpha_m]}{\bigg|} \det \mathcal{W}.
$$
\end{proof}

\begin{proof}[Proof of Theorem \ref{IntegerFactortheorem}]
The integer factor \eqref{integerfactor} follows from Theorem \ref{polynomialfactortheorem} by choosing $\overline{\alpha}=(a_1, \ldots, a_m)^T$ and noticing that $\widehat{D}(\overline{a}) \underset{\mathbb{Z}}{\big|} D(\overline{a})$.
\end{proof}

\subsection{Rank}

In order to benefit from Theorem \ref{IntegerFactortheorem}, we should have $\rank \mathcal{V}(\overline{a})=M$ over $\mathbb{Z}$.
Proving this seems to be out of reach at the moment. (The polynomials in 
%Appendix 
\ref{minorappendix} illustrate the problem: some of the minors could clearly vanish when the alphas are given integer values, so one would have to show that this does not happen to all the minors simultaneously.)
However, as a step towards the solution we shall show that  
$\rank \mathcal{V}(\overline{\alpha}) = M$ over the polynomial ring $\Q [\alpha_1, \ldots, \alpha_m]$. 

\begin{lemma}\label{ranklemma2}
We have $\rank \mathcal{V}(\overline{\alpha}) = M$ over the polynomial ring $\Q [\alpha_1, \ldots, \alpha_m]$.
\end{lemma}

\begin{proof}
We show that one of the $M \times M$ minors of $\mathcal{V}(\overline{\alpha})$ is not the zero polynomial, in this case the rightmost one:
\begin{equation}\label{detM0}
\begin{vmatrix}
\binom{L + 1}{L-(M-1)} \alpha_1^{M-1} & \cdots & \binom{L + 1}{L - 1} \alpha_1 & \binom{L +1}{L} \\
\binom{L + 2}{L-(M-1)} \alpha_1^{M-1} & \cdots & \binom{L + 2}{L - 1} \alpha_1 & \binom{L +2}{L} \\
\vdots & \ddots & \vdots & \vdots \\
\binom{L + \nu_1}{L-(M-1)} \alpha_1^{M-1} & \cdots & \binom{L + \nu_1}{L - 1} \alpha_1 & \binom{L +\nu_1}{L} \\
\vdots & \ddots & \vdots & \vdots \\

\vdots & \ddots & \vdots & \vdots \\
\binom{L + 1}{L-(M-1)} \alpha_m^{M-1} & \cdots & \binom{L + 1}{L - 1} \alpha_m & \binom{L +1}{L} \\

\vdots & \ddots & \vdots & \vdots \\

\binom{L + \nu_m}{L-(M-1)} \alpha_m^{M-1} & \cdots & \binom{L + \nu_m}{L - 1} \alpha_m & \binom{L +\nu_m}{L} \\
\end{vmatrix}.
\end{equation}
 The monomial
$$
\alpha_1^{\binom{\nu_1}{2} + \nu_1 \nu_2 + \nu_1 \nu_3 + \ldots + \nu_1 \nu_m} 
\alpha_2^{\binom{\nu_2}{2} + \nu_2 \nu_3 + \nu_2 \nu_4 + \ldots + \nu_2 \nu_m} \cdots 
\alpha_{m-1}^{\binom{\nu_{m-1}}{2} + \nu_{m-1} \nu_m} \alpha_m^{\binom{\nu_m}{2}}
$$
arises uniquely from the block diagonal when using the generalised minor expansion with $\nu_j \times \nu_j$, $j=1, \ldots, m$, 
minors to expand \eqref{detM0}.
We show that its coefficient, denoted by $G$, is non-zero, therefore implying that the determinant \eqref{detM0} is a non-zero polynomial.

The generalised minor expansion (Lemma \ref{manyblocksexpansion}) gives
\begin{equation*}
\begin{split}
&G \cdot 
\alpha_1^{\binom{\nu_1}{2} + \nu_1 \nu_2 + \nu_1 \nu_3 + \ldots + \nu_1 \nu_m} 
\alpha_2^{\binom{\nu_2}{2} + \nu_2 \nu_3 + \nu_2 \nu_4 + \ldots + \nu_2 \nu_m} \cdots 
\alpha_{m-1}^{\binom{\nu_{m-1}}{2} + \nu_{m-1} \nu_m} \alpha_m^{\binom{\nu_m}{2}} \\
= \;&\prod_{j=1}^m
\begin{vmatrix}
\binom{L + 1}{L-M+(\nu_1+\ldots+\nu_{j-1}+1)} \alpha_j^{M-(\nu_1+\ldots+\nu_{j-1}+1)} & \cdots & \binom{L + 1}{L - M+(\nu_1+\ldots+\nu_j)} \alpha_j^{M-(\nu_1+\ldots+\nu_j)}\\
\binom{L + 2}{L-M+(\nu_1+\ldots+\nu_{j-1}+1)} \alpha_j^{M-(\nu_1+\ldots+\nu_{j-1}+1)} & \cdots & \binom{L + 2}{L -M+(\nu_1+\ldots+\nu_j)} \alpha_j^{M-(\nu_1+\ldots+\nu_j)}\\
\vdots & \ddots & \vdots \\
\binom{L + \nu_j}{L-M+(\nu_1+\ldots+\nu_{j-1}+1)} \alpha_j^{M-(\nu_1+\ldots+\nu_{j-1}+1)} & \cdots & \binom{L + \nu_j}{L - M+(\nu_1+\ldots+\nu_j)} \alpha_j^{M-(\nu_1+\ldots+\nu_j)}\\
\end{vmatrix}_{\nu_j \times \nu_j} \\
= \;
&\left( \prod_{j=1}^m \alpha_j^{\binom{\nu_j}{2} + \nu_j (\nu_{j+1}+ \ldots + \nu_m)} \right) \cdot \prod_{j=1}^m g_j,
\end{split}
\end{equation*}
where
\begin{equation}\label{Condetprod}
g_j :=
\begin{vmatrix}
\binom{L + 1}{L-M+(\nu_1+\ldots+\nu_{j-1}+1)} & \cdots & \binom{L + 1}{L - M+(\nu_1+\ldots+\nu_j)} \\
\binom{L + 2}{L-M+(\nu_1+\ldots+\nu_{j-1}+1)} & \cdots & \binom{L + 2}{L -M+(\nu_1+\ldots+\nu_j)} \\
\vdots & \ddots & \vdots \\
\binom{L + \nu_j}{L-M+(\nu_1+\ldots+\nu_{j-1}+1)} & \cdots & \binom{L + \nu_j}{L - M+(\nu_1+\ldots+\nu_j)} \\
\end{vmatrix}_{\nu_j \times \nu_j}.
\end{equation}
(Note that $\nu_1 + \ldots + \nu_{j-1} = 0$ when $j=1$.)
We get an expression for $G$ as a product of determinants: $G = \prod_{j=1}^m g_j$. To see the structure of these determinants more clearly, we write the binomial coefficients in different form:
$$
\binom{n}{k} = \frac{(n]_k}{k!},
$$
where
$$
(n]_0 := 1; \quad (n]_k := n (n-1) \cdots (n-k+1), \quad k =1, \ldots, n.
$$
As stated before, the falling factorial $(n]_k$ is a polynomial in $n$ of degree $k$. Just as in the proof of Lemma \ref{ranklemma}, the determinants in \eqref{Condetprod} are thus essentially polynomial Vandermonde matrices, the polynomials being $(x]_0, (x]_1, \ldots, (x]_{\nu_j-1}$. Again we may exploit Corollary \ref{Kratt} and get
\begin{equation*}
\begin{split}
g_j = &\begin{vmatrix}
\binom{L + 1}{L-M+(\nu_1+\ldots+\nu_{j-1}+1)} & \cdots & \binom{L + 1}{L - M+(\nu_1+\ldots+\nu_j)} \\
\binom{L + 2}{L-M+(\nu_1+\ldots+\nu_{j-1}+1)} & \cdots & \binom{L + 2}{L -M+(\nu_1+\ldots+\nu_j)} \\
\vdots & \ddots & \vdots \\
\binom{L + \nu_j}{L-M+(\nu_1+\ldots+\nu_{j-1}+1)} & \cdots & \binom{L + \nu_j}{L - M+(\nu_1+\ldots+\nu_j)} \\
\end{vmatrix}_{\nu_j \times \nu_j} \\
= \;&\frac{1}{(L-M+(\nu_1+\ldots+\nu_{j-1}+1))!} \cdots \frac{1}{(L - M+(\nu_1+\ldots+\nu_j))!} \cdot \\
&\begin{vmatrix}
(L+1]_{L-M+(\nu_1+\ldots+\nu_{j-1}+1)} & \cdots & (L+1]_{L - M+(\nu_1+\ldots+\nu_j)} \\
(L+2]_{L-M+(\nu_1+\ldots+\nu_{j-1}+1)} & \cdots & (L+2]_{L - M+(\nu_1+\ldots+\nu_j)} \\
\vdots & \ddots & \vdots \\
(L+\nu_j]_{L-M+(\nu_1+\ldots+\nu_{j-1}+1)} & \cdots & (L+\nu_j]_{L - M+(\nu_1+\ldots+\nu_j)} \\
\end{vmatrix}_{\nu_j \times \nu_j} \\
= \;&\frac{1}{(L-M+(\nu_1+\ldots+\nu_{j-1}+1))!} \cdots \frac{1}{(L - M+(\nu_1+\ldots+\nu_j))!} \cdot \\
&(L+1]_{L-M+(\nu_1+\ldots+\nu_{j-1}+1)} \cdots (L+\nu_j]_{L-M+(\nu_1+\ldots+\nu_{j-1}+1)} \cdot \\
&\begin{vmatrix}
1 & (M-(\nu_1 + \ldots + \nu_{j-1})]_1 & \cdots & (M-(\nu_1 + \ldots + \nu_{j-1})]_{\nu_j-1} \\
1 & (M+1-(\nu_1 + \ldots + \nu_{j-1})]_1 & \cdots & (M+1-(\nu_1 + \ldots + \nu_{j-1})]_{\nu_j-1} \\
\vdots & & \ddots & \vdots \\
1 & (M+\nu_j-1-(\nu_1 + \ldots + \nu_{j-1})]_1 & \cdots & (M+\nu_j-1-(\nu_1 + \ldots + \nu_{j-1})]_{\nu_j-1} \\
\end{vmatrix}_{\nu_j \times \nu_j} \\
= \;&\frac{1}{(L-M+(\nu_1+\ldots+\nu_{j-1}+1))!} \cdots \frac{1}{(L - M+(\nu_1+\ldots+\nu_j))!} \cdot \\
&(L+1]_{L-M+(\nu_1+\ldots+\nu_{j-1}+1)} \cdots (L+\nu_j]_{L-M+(\nu_1+\ldots+\nu_{j-1}+1)} \cdot (\nu_j -1)! (\nu_j -2)! \cdots 1! \neq 0\\
%\neq \;& 0
\end{split}
\end{equation*}
for all $j=1, \ldots, m$. Hence also $G = \prod_{j=1}^m g_j \neq 0$.
\end{proof}

\subsection{Twin type II Pad\'e approximants}

Bombieri and Vaaler's theorem \ref{Bombieri-Vaaler} made us look for a big common factor from the $M \times M$ minors of the matrix $\mathcal{V}(\overline{a}) \in \mathcal{M}_{M \times (L+1)}(\Z)$.
This common factor is a result from the general polynomial common factor \eqref{polynomialfactor} of the $M \times M$ minors of the matrix $\mathcal{V}(\overline{\alpha})$ proved in Theorem \ref{polynomialfactortheorem}.
As it happens, the explicit solution to the Pad\'e approximation problem \eqref{TwinHermite2} requires calculating those same minors of $\mathcal{V}(\overline{\alpha})$ in the special case $M=L$. 
(This has been explained in Section \ref{Cramersrule}, where the procedure was successfully carried out in the 'tame' case.)

So, as a special case ($\nu_i=l_i$, $M=L$) of Theorem \ref{polynomialfactortheorem} and Lemma \ref{ranklemma2} we get Theorem \ref{UnKnown}, a partial solution to the twin problem \eqref{TwinHermite2}.

\begin{proof}[Proof of Theorem \ref{UnKnown}]
Recall again that by Cramer's rule, equation \eqref{unknowneq} has a solution
\begin{equation*}
[c_0, c_1, \ldots, c_L] = 
\left[ \mathcal{V}[0] : - \mathcal{V}[1] : \ldots : (-1)^L \mathcal{V}[L] \right],
\end{equation*}
where one of the $L \times L$ minors $\mathcal{V}[i] \in \Z[\alpha_1, \ldots, \alpha_m]$ is a non-zero polynomial by Lemma \ref{ranklemma2}. 
Because of the homogeneous coordinates, we may divide by the common factor 
$T\!\left(\overline{l},\overline{\alpha}\right) = \alpha_1^{\binom{l_1}{2}} \cdots \alpha_m^{\binom{l_m}{2}} \prod_{1 \le i < j \le m} (\alpha_i - \alpha_j)^{\min \left\{ l_i^2,\, l_j^2 \right\}}$
(Theorem \ref{polynomialfactortheorem} with $\nu_i=l_i$), so that
$$
c_i = \frac{(-1)^i\mathcal{V}[i]}{T\!\left(\overline{l},\overline{\alpha}\right)} =: \tau_{i}(\overline{l},\overline{\alpha}), \quad i=0,1,\ldots,L.
$$
Now \eqref{Bnu0-2} gives
$$
B_{\overline{l},0}(t,\overline{\alpha}) = \sum_{i=0}^L \frac{L!}{i!} \tau_{i}\!\left(\overline{l},\overline{\alpha}\right) t^i \in \Z[t,\overline{\alpha}].
$$
\end{proof}

The proof of Theorem \ref{UnKnown} could be carried out in the same way as the proof of Theorem \ref{Known}, 
leaving out the term $\frac{L!}{h!}$ in \eqref{Bnu0-2}. This would result in a slightly different matrix, 
but essentially the same common factor. This is due to the connection between the minors of these matrices, 
apparent in \eqref{connection1}. Since our starting point was the Bombieri-Vaaler version of Siegel's lemma,
we had to manipulate the equations so that the coefficient matrix has integer elements. 
This technicality only is the reason to the presence of $\frac{L!}{i!}$ in Theorem \ref{UnKnown}.

Near the diagonal the common factor $T\!\left(\overline{l},\overline{\alpha}\right)$ is big, implying that the coefficients
$\tau_{i}\!\left(\overline{l},\overline{\alpha}\right)$ are relatively small. 
However, the next example of the case $m=2$, $l_0=0$, $l_1=1$, $l_2=3$ in Table \ref{table1} already shows the big difference between the Pad\'e polynomials 
$A_{\overline{l},0}(t,\overline{\alpha})=\sum_{i=0}^L b_i t^i = \sum_{h=0}^L (L_0-h)! \sigma_{L_0-h}\! \left(\overline l,\overline{\alpha}\right) t^h$ and  
$B_{\overline{l},0}(t,\overline{\alpha}) = \sum_{i=0}^L \frac{L!}{i!} \tau_{i}\!\left(\overline{l},\overline{\alpha}\right) t^i$, 
illustrating our decision to call the latter case 'wild'.

\begin{table}[h]
\caption{Case $m=2$, $l_0=0$, $l_1=1$, $l_2=3$.}
\label{table1}
\centering
{\renewcommand{\arraystretch}{1.2}
\begin{tabular}{|l|l|l|}
\hline
$i$ & $b_i$ & $\frac{L!}{i!} \tau_{i}\!\left(\overline{l},\overline{\alpha}\right)$ \\
\hline
0 &  24                               & $24 \cdot 35(5\alpha_1^2-5\alpha_1\alpha_2+\alpha_2^2)$ \\
1 & $-6(\alpha_1+3\alpha_2)$          & $-24 \cdot 5(7\alpha_1^3+7\alpha_1^2\alpha_2-13\alpha_1\alpha_2^2+3\alpha_2^3)$ \\
2 & $6\alpha_2(\alpha_1+\alpha_2)$    & $12\cdot 5\alpha_2(7\alpha_1^3-3\alpha_1^2\alpha_2-3\alpha_1\alpha_2^2+\alpha_2^3)$ \\
3 & $-\alpha_2^2(3\alpha_1+\alpha_2)$ & $-4\alpha_2^2(3\alpha_1-\alpha_2)(7\alpha_1^2-4\alpha_1\alpha_2-\alpha_2^2)$\\
4 & $\alpha_1 \alpha_2^3$             & $\alpha_1\alpha_2^3(7\alpha_1^2-8\alpha_1\alpha_2+2\alpha_2^2)$ \\
\hline
\end{tabular}}
\end{table}

\appendix
\section{Some examples of minors of the matrix $\mathcal{V}(\overline{\alpha})$}\label{minorappendix}

Tables \ref{table2} and \ref{table3} show some computer calculated examples in the case $\nu_i=l_i$, $M=L$ with small values of $m$ and $l_i$. Here the gcd contains an integer part in addition to the primitive polynomial gcd earlier denoted by $\widehat{D}(\overline{\alpha})$.

The gcd of the $L \times L$ minors is presented in the second column, and the minors divided by the gcd are listed in the third column. 
The order of the minors is $\mathcal{V}[0], \mathcal{V}[1], \ldots, \mathcal{V}[L]$.

Factorisation was performed with \emph{Mathematica}.

\begin{table}[h]
\caption{Some cases with $m=2$.}
\label{table2}
\centering
{\renewcommand{\arraystretch}{1.2}
\begin{tabular}{|l|c|l|}
\hline
Case & GCD & $\mathcal{V}[i] / \mathrm{GCD}$, $i=0,1,\ldots,L$ \\
\hline
$m=2,$ & $3(\alpha_1-\alpha_2)$ & $6$, \\
$L=2,$ & & $2(\alpha_1+\alpha_2)$,\\
$l_1 = l_2 = 1$ & & $\alpha_1 \alpha_2$\\
\hline
$m=2,$ & $2\alpha_2(\alpha_1-\alpha_2)$ & $20 (2\alpha_1-\alpha_2)$,\\
$L=3,$ & & $2(5\alpha_1^2+5\alpha_1\alpha_2-4\alpha_2^2),$ \\
$l_1 = 1$, $l_2 = 2$ & & $2\alpha_2(5\alpha_1^2-\alpha_1\alpha_2-{\alpha_2}^2)$, \\
& & $\alpha_1\alpha_2^2(5\alpha_1-3\alpha_2)$ \\
\hline
$m=2,$ & $25\alpha_1\alpha_2 (\alpha_1-\alpha_2)^4$ & $30$, \\
$L=4,$ & & $10(\alpha_1+\alpha_2)$, \\
$l_1 = l_2 = 2$ & & $2(\alpha_1^2+4\alpha_1\alpha_2+\alpha_2^2)$, \\
& & $3\alpha_1\alpha_2(\alpha_1+\alpha_2)$, \\
& & $2\alpha_1^2 \alpha_2^2$ \\
\hline
$m=2,$ & $5\alpha_2^3 (\alpha_1-\alpha_2)$ & $35(5\alpha_1^2-5\alpha_1\alpha_2+\alpha_2^2)$,\\
$L=4,$ & & $5(7\alpha_1^3+7\alpha_1^2\alpha_2-13\alpha_1\alpha_2^2+3\alpha_2^3)$, \\
$l_1 = 1$, $l_2 = 3$ & & $5\alpha_2(7\alpha_1^3-3\alpha_1^2\alpha_2-3\alpha_1\alpha_2^2+\alpha_2^3)$, \\
& & $\alpha_2^2(3\alpha_1-\alpha_2)(7\alpha_1^2-4\alpha_1\alpha_2-\alpha_2^2)$, \\
& & $\alpha_1\alpha_2^3(7\alpha_1^2-8\alpha_1\alpha_2+2\alpha_2^2)$ \\
\hline
$m=2,$ &  $3 \alpha_2^6 (\alpha_1-\alpha_2)$ & $252 (2\alpha_1 - \alpha_2) (7 \alpha_1^2 - 7\alpha_1\alpha_2 + \alpha_2^2)$, \\
$L=5,$ & & $14 (42 \alpha_1^4 + 42 \alpha_1^3 \alpha_2 - 133 \alpha_1^2 \alpha_2^2 + 67 \alpha_1 \alpha_2^3 - 8 \alpha_2^4)$, \\
$l_1 = 1$, & & $14 \alpha_2 (42 \alpha_1^4 - 28 \alpha_1^3 \alpha_2 + 22 \alpha_1 \alpha_2^3 - 3 \alpha_2^4)$, \\
$l_2 = 4$ & & $3 \alpha_2^2 (126 \alpha_1^4 - 154 \alpha_1^3 \alpha_2 + 21 \alpha_1^2 \alpha_2^2 + 21 \alpha_1 \alpha_2^3 - 4 \alpha_2^4)$, \\
& & $2 \alpha_2^3 (84 \alpha_1^4 - 126 \alpha_1^3 \alpha_2 + 49 \alpha_1^2 \alpha_2^2 - \alpha_1 \alpha_2^3 - \alpha_2^4)$, \\
& & $\alpha_1\alpha_2^4 (42 \alpha_1^3 -70 \alpha_1^2 \alpha_2 + 35 \alpha_1 \alpha_2^2 - 5 \alpha_2^3)$ \\
\hline
$m=2,$ & $7 \alpha_1 \alpha_2^3 (\alpha_1-\alpha_2)^4$ & $420 (7 \alpha_1^2 -7\alpha_1\alpha_2 + 2 \alpha_2^2)$, \\
$L=5,$ & & $105 (8 \alpha_1^3 + 2 \alpha_1^2 \alpha_2 - 8 \alpha_1 \alpha_2^2 + 3 \alpha_2^3)$, \\
$l_1 = 2$, & & $10 (14 \alpha_1^4 + 56 \alpha_1^3 \alpha_2 - 40 \alpha_1^2 \alpha_2^2 - 9 \alpha_1 \alpha_2^3 + 9 \alpha_2^4)$, \\
$l_2 = 3$ & & $15\alpha_2 (14 \alpha_1^4 + 8 \alpha_1^3 \alpha_2 - 17 \alpha_1^2 \alpha_2^2 + 4 \alpha_1 \alpha_2^3 + \alpha_2^4)$, \\
& & $6 \alpha_1 \alpha_2^2 (28 \alpha_1^3 - 16 \alpha_1^2 \alpha_2 - 7 \alpha_1 \alpha_2^2 + 5 \alpha_2^3)$, \\
& & $5 \alpha_1^2 \alpha_2^3 (14 \alpha_1^2 - 16\alpha_1\alpha_2 + 5 \alpha_2^2)$ \\
\hline
$m=2,$ & $4116 \alpha_1^3 \alpha_2^3 (\alpha_1-\alpha_2)^9$ & $84$, \\
$L=6,$ & & $28(\alpha_1+\alpha_2)$, \\
$l_1 = l_2 = 3$ & & $7 (\alpha_1^2 + 3\alpha_1\alpha_2 + \alpha_2^2)$, \\
& & $(\alpha_1+\alpha_2) (\alpha_1^2 + 8\alpha_1\alpha_2 + \alpha_2^2)$, \\
& & $2 \alpha_1\alpha_2 (\alpha_1^2 + 3\alpha_1\alpha_2 + \alpha_2^2)$, \\
& & $2 \alpha_1^2 \alpha_2^2 (\alpha_1+\alpha_2)$, \\
& & $\alpha_1^3 \alpha_2^3$ \\
\hline
\end{tabular}}
\end{table}

\begin{table}[h]
\caption{Some cases with $m=3$.}
\label{table3}
\centering
{\renewcommand{\arraystretch}{1.2}
\begin{tabular}{|l|c|l|}
\hline
Case & GCD & $\mathcal{V}[i] / \mathrm{GCD}$, $i=0,1,\ldots,L$ \\
\hline
$m=3,$ & $8 (\alpha_1-\alpha_2) \cdot$ & $12$, \\
$L=3$, & $(\alpha_1-\alpha_3) (\alpha_2-\alpha_3)$ & $3(\alpha_1+\alpha_2+\alpha_3)$, \\
$l_1 = l_2 =$ & & $2(\alpha_1\alpha_2+\alpha_1\alpha_3+\alpha_2\alpha_3)$, \\
$ l_3 = 1$& & $3\alpha_1\alpha_2\alpha_3$ \\
\hline
$m=3$, & $25\alpha_3 (\alpha_1-\alpha_2) \cdot$ & $10 (10\alpha_1\alpha_2 - 5\alpha_1\alpha_3 - 5\alpha_2\alpha_3 + 3\alpha_3^2)$, \\
$l_1 = l_2 = 1$, & $(\alpha_1-\alpha_3)(\alpha_2-\alpha_3)$ & $10 (2 \alpha_1^2 \alpha_2 + 2\alpha_1\alpha_2^2 - \alpha_1^2 \alpha_3 +\alpha_1\alpha_2\alpha_3$ \\
$l_3 = 2$ & & $ - \alpha_2^2 \alpha_3 - \alpha_1\alpha_3^2 - \alpha_2\alpha_3^2 + \alpha_3^3)$,\\
$L=4$, & & $2 (5 \alpha_1^2 \alpha_2^2 + 5 \alpha_1^2 \alpha_2\alpha_3 + 5 \alpha_1\alpha_2^2 \alpha_3 - 4 \alpha_1^2 \alpha_3^2$ \\
& & $- 4 \alpha_1\alpha_2 \alpha_3^2 - 4 \alpha_2^2 \alpha_3^2 + \alpha_1\alpha_3^3 + \alpha_2\alpha_3^3 + \alpha_3^4)$, \\
& & $\alpha_3 (15 \alpha_1^2 \alpha_2^2 - 3 \alpha_1^2 \alpha_2\alpha_3 - 3\alpha_1 \alpha_2^2 \alpha_3 - 3 \alpha_1^2 \alpha_3^2$ \\
& & $- \alpha_1\alpha_2\alpha_3^2 - 3\alpha_2^2 \alpha_3^2 + 2\alpha_1\alpha_3^3 + 2\alpha_2\alpha_3^3)$, \\
& & $2\alpha_1\alpha_2\alpha_3^2 (5\alpha_1\alpha_2 - 3\alpha_1\alpha_3 - 3\alpha_2\alpha_3 + 2\alpha_3^2)$ \\
\hline
$m=3$, & $19208 \alpha_1\alpha_2\alpha_3 \cdot$ & $\ldots$ \\
$L=6$, & $(\alpha_1-\alpha_2)^4 (\alpha_1-\alpha_3)^4 \cdot$ & \\
$l_1 = l_2 =$ & $(\alpha_2-\alpha_3)^4$ & \\
$l_3 = 2$ & & \\
\hline
\end{tabular}}
\end{table}

\clearpage

\section*{Acknowledgements}

The authors are indebted to the anonymous referee whose comments helped to improve the presentation of this article.

\section*{Funding}

The work of Louna Sepp\"al\"a was supported by the Magnus Ehrnrooth Foundation.

\end{document}